\documentclass[11pt,twoside,reqno]{amsart}
\usepackage{amsfonts}
\usepackage{epsfig,graphics}
\usepackage{amssymb}
\usepackage{amscd}
\usepackage{amsbsy}

\newtheorem{theorem}{Theorem}[section]

\newtheorem{proposition}[theorem]{Proposition}
       
\newtheorem{definition}[theorem]{Definition}

\newtheorem{corollary}[theorem]{Corollary}

\newtheorem{question}{Question}

\theoremstyle{definition}
 \newtheorem{defin}{Definition}[section]
 \newtheorem{remark}{Remark}[section]

\newcommand{\Ad}{\text{Ad }}

\newcommand{\Ker}{\text{ker }}
\newcommand{\im}{\text{im }}

\newcommand{\SL}{\text{SL}}

\newcommand{\BR}{{\bf R}}
\newcommand{\BC}{{\bf C}}

\newcommand{\lieg}{{\mathfrak{g}}}

\newcommand{\liep}{{\mathfrak{p}}}
\newcommand{\lieq}{{\mathfrak{q}}}

\newcommand{\alggl}{{\mathfrak{g}\mathfrak{l}}}
\newcommand{\algsl}{{\mathfrak{s}\mathfrak{l}}}

\newcommand{\dds}{\frac{\mathrm{d}}{\mathrm{d}s}}

\begin{document}
\pagenumbering{arabic}
\title[Strongly essential flows on parabolic geometries]{Strongly essential flows on irreducible parabolic geometries}
\author{Karin Melnick and Katharina Neusser}
\date{\today}
\begin{abstract}
We study the local geometry of irreducible parabolic geometries admitting strongly essential flows; these are flows by local automorphisms 
with higher-order fixed points. We prove several new rigidity results, and recover some old ones 
for projective and conformal structures, which show that in many cases the existence of a strongly essential flow implies local flatness of the geometry on
an open set having the fixed point in its closure. For almost c-projective and almost quaternionic structures we can moreover show flatness of 
the geometry on a neighborhood of the fixed point.
\end{abstract}

\maketitle

\section{Introduction}
Irreducible parabolic geometries are a family of differential geometric structures including conformal semi-Riemannian, projective, almost c-projective, almost quaternionic, and almost Grassmannian structures. They are Cartan geometries infinitesimally modeled on homogeneous projective varieties $G/P$, where $G$ is a semisimple Lie group and $P < G$ is a parabolic subgroup with abelian unipotent radical  (see Section \ref{sec.def.note} below for definitions and references). 
This article continues the study, initiated in \cite{cap.me.parabolictrans}, of the local geometry of parabolic geometries in the presence of a flow by \emph{strongly essential automorphisms}. For irreducible parabolic geometries, these are $1$-parameter families $\{ \varphi^t \}$ of local automorphisms fixing a point $x_0$ and having trivial differential at $x_0$. The corresponding vector field $\eta$ is called \emph{strongly essential} and is said to have a \emph{higher-order zero} at $x_0$. 

Nagano and Ochiai proved, using the canonical Cartan connection of a projective structure, that a projective manifold admitting a projective vector field with higher-order zero at $x_0$ is projectively flat in a neighorhood of $x_0$---that is, locally projectively equivalent to projective space \cite{nagano.ochiai.proj}. The analogous result was obtained for conformal semi-Riemannian structures by the first author and Frances in \cite{fm.champsconfs}, although in this case, without assuming real analyticity of the structure, the semi-Riemannian conformal manifold may only be conformally flat on an open set $U$ with the higher-order zero $x_0$ in its closure; for semi-Riemannian conformal structures the open set $U$ can be the interior or exterior of the light cone through $x_0$, or a more complicated open union of semi-cones.

The paper \cite{cap.me.parabolictrans} unifies and generalizes the methods of \cite{nagano.ochiai.proj} and \cite{fm.champsconfs} to general parabolic geometries, 
taking advantage of the relation between parabolic geometries and the representation-theoretic properties of the model pair $(\lieg,P)$ (see also \cite{cap.me.proj.conf}). 
The key product of this relation is the harmonic curvature (see Section \ref{sec.def.note} below), which arises from a completely reducible quotient of the curvature representation of the Cartan geometry, yet is the obstruction to vanishing of the full curvature on an open set. The general results of \cite{cap.me.parabolictrans} associate to the fixed point $x_0$ of a strongly essential flow an explicit, algebraically defined family of curves emanating from $x_0$ and give restrictions, in terms of the algebraic properties of the pair $(\lieg, P)$, on values of the harmonic curvature along these special curves. The general criteria of \cite{cap.me.parabolictrans} are applied to obtain rigidity results for strongly essential flows for a variety of parabolic geometries, including almost-quaternionic, contact parabolic, and strictly pseudoconvex CR structures.  In many cases, however, the methods of \cite{cap.me.parabolictrans} do not give full vanishing of the harmonic curvature along the special curves, but only of some components; moreover, in each case they involve a detailed analysis of the harmonic curvature module.

In this article, we develop new techniques showing vanishing of the harmonic curvature along these special curves, which do not require precise knowledge of the harmonic curvature representation for each type of structure. We establish a general result for irreducible parabolic geometries infinitesimally modeled on $G/P$ with $G$ simple, which shows, given a strongly essential flow, that the harmonic curvature vanishes along the special curves through the fixed point $x_0$, and moreover that the full Cartan curvature vanishes at $x_0$. 
We use this theorem to prove several new rigidity results for irreducible parabolic geometries. In fact, we show in many cases that the existence of a strongly essential flow implies that the geometry is locally flat on an open set $U$ having $x_0 \in \overline{U}$.  In some cases, such as for almost c-projective structures, we can strengthen this local flatness to a neighborhood of $x_0$.  We also recover some of the above mentioned results of \cite{fm.champsconfs} on semi-Riemannian conformal structures and the results of \cite{nagano.ochiai.proj} on projective structures, and improve the results of \cite{cap.me.parabolictrans} on almost quaternionic structures and almost Grassmannian structures of type $(2,n)$.

\subsection{Results}
\label{sec.intro.results}

This section begins with the definitions relevant to our main results.  Background on irreducible parabolic geometries and further definitions are given in Section \ref{sec.def.note} below.  We will assume throughout that $M$ is a smooth, connected manifold endowed with a smooth geometric structure or equivalent smooth Cartan geometry.

For a parabolic subgroup $P$ of a semisimple Lie group $G$, denote by $\liep$ and $\lieg$ their respective Lie algebras.
Let $(B\stackrel{\pi}{\rightarrow}M,\omega)$ be a normal Cartan geometry of irreducible parabolic type $(\lieg, P)$.  We denote by $\mathfrak{inf}(M)$ the algebra of smooth vector fields $\eta \in \mathfrak{X}(M)$ along which the flow $\{ \varphi^t_\eta \} $, where defined, is by automorphisms of $(B \stackrel{\pi}{\rightarrow}M,\omega)$; one need not assume $\eta$ complete. The unique lift of $\eta$ to $B$ is denoted $\tilde{\eta}$.

\begin{defin}
\label{def.isotropy}
Given $\eta \in \mathfrak{inf}(M)$ vanishing at $x_0 \in M$, the \emph{isotropy} of $\eta$ with respect to $b_0 \in \pi^{-1}(x_0)$ is $\omega_{b_0}(\tilde{\eta}) \in \liep$.
\end{defin}

The $P$-equivariance of the Cartan connection $\omega$ implies that a different choice of $b_0 \in \pi^{-1}(x_0)$ yields a conjugate value for the isotropy.  


Remark that $\eta$ is strongly essential if and only if its isotropy at $x_0$ with respect to any $b_0$ belongs to the nilradical $\liep_+$ of $\liep$.  Via the duality $\liep_+ \cong (\lieg / \liep)^*$ of $P$-modules, the isotropy of a strongly essential $\eta$ determines $\alpha \in T_{x_0}^*M$, because $T^*M  \cong B \times_P (\lieg/\liep)^*$.  We will also call $\alpha$ the isotropy of $\eta$ at $x_0$.



Given $b_0 \in B$ and $X \in \lieg$, there is a curve through $b_0$ 
$$ \tilde{\gamma}_X(s) = \exp_{b_0}(sX), \qquad s \in (- \epsilon, \epsilon)$$

defined for some $\epsilon > 0$; here $\exp$ is the Cartan geometry exponential map sending $(b_0, X)$ to the image of $b_0$ under the time-one flow along the $\omega$-constant vector field determined by $X$ (see Definition \ref{def.exp.map} below).  For $\tilde{\gamma}_X$ as above, $\gamma_X = \pi \circ \tilde{\gamma}_X$ is called an \emph{exponential curve} through $x_0 = \pi(b_0)$.


Our general results provide conditions under which the harmonic curvature vanishes along a family of exponential curves through the fixed point $x_0$ of a strongly essential flow, corresponding to the following subset of $\lieg$.  The subspace $\lieg_-$ is a vector space complement to $\liep$ in $\lieg$, as defined as in Section \ref{parabolic_geometries} below.

\begin{defin}
\label{def.counterpart}
For $0\neq Z \in \liep_+$ denote $T(Z)$ the set of $X \in \lieg_-$ such that $Z, A = [Z,X],$ and $X$ form an $\mathfrak{sl}_2$-triple---that is,
$$ [A,Z] = 2Z \qquad \mbox{and} \qquad [A,X] = - 2X$$
\end{defin}   

When $(\lieg,P)$ is of irreducible type, then $T(Z)$ is always nonempty; this is a consequence of the Jacobson-Morozov Theorem and Proposition 2.16 of \cite{cap.me.parabolictrans}, which says that any $X$ generating an $\algsl_2$-triple with $Z$ is equivalent modulo $\liep$ to an element of $T(Z)$.  
Let $Z \in \liep_+$ be the isotropy of a strongly essential $\eta \in \mathfrak{inf}(M)$ with respect to $b_0 \in \pi^{-1}(x_0)$.  Via the isomorphism $\lieg_- \cong \lieg/\liep$ (see Section \ref{adapted_sl2_triples}), and the bundle isomorphism $TM \cong B \times_P \lieg / \liep$, the set $T(Z)$ corresponds to a subset of $T_{x_0} M$, which we denote $T(\alpha)$.  The corresponding collection of exponential curves $\gamma_X$ through $x_0$ is denoted $\mathcal{T}(\alpha)$.  The subsets $T(\alpha)$ and the collection of curves $\mathcal T(\alpha)$ depend only on $\alpha$, and not on the choice of $b_0\in\pi^{-1}(x_0)$.

The Cartan curvature, viewed as a vector-valued function on $B$, will be denoted $\kappa$ below, while the harmonic curvature will be $\hat{\kappa}$ (see Section \ref{sec.def.note} below).  Given a strongly essential flow with fixed point $x_0$ and isotropy $Z$ with respect to $b_0 \in \pi^{-1}(x_0)$, our first general result, Proposition \ref{prop.s.to.the.k}, establishes a polynomial form for $\hat{\kappa}$ along the curves $\{ \tilde{\gamma}_X \,: \, X \in T(Z) \}$ through $b_0$.  This proposition is key in the proof of our main theorem, which follows:

\begin{theorem}
\label{thm.harmonic.vanishing}
Suppose $(B\stackrel{\pi}{\rightarrow} M,\omega)$ is a normal irreducible parabolic geometry of type $(\lieg,P)$ with $\lieg$ simple.  
Let $\eta\in\mathfrak{inf}(M)$ be a nontrivial infinitesimal automorphism with higher-order zero at $x_0\in M$ and isotropy $\alpha \in T_{x_0}^*M$.  Then
\begin{itemize}
\item $\hat\kappa$ vanishes along all curves in $\mathcal{T}(\alpha).$
\item $\kappa(x_0) = 0.$
\end{itemize}
\end{theorem}
The equivariant functions $\kappa$ and $\hat{\kappa}$ correspond to sections of associated vector bundles.  The expression $\kappa(x_0) = 0$ denotes vanishing of this section at $x_0$, and similarly for $\hat{\kappa}$.

We apply Theorem \ref{thm.harmonic.vanishing} to obtain two general results.  The first, Proposition \ref{Case_isolated_fixed_point}, establishes that in the presence 
of a strongly essential flow with \emph{smoothly isolated} higher-order fixed point (see Definition \ref{defin_isol_zero} below) the curvature always vanishes on a nonempty open set with the fixed point in its closure.  Proposition \ref{Case_isolated_fixed_point} applies in particular to any strongly essential flow of a projective, almost c-projective, or almost quaternionic structure.
As for projective structures in \cite{nagano.ochiai.proj}, we can improve the curvature vanishing to a neighborhood of the fixed point for these latter two structures; for example:

\begin{theorem}
\label{thm.almost.cproj}
Let $M^{2n}, n \geq 2,$ be endowed with a smooth almost c-projective structure.  Suppose $0 \neq \eta \in \mathfrak{inf}(M)$ is a c-projective vector field with a 
higher-order zero at $x_0 \in M$.  Then there exists a neighborhood of $x_0$ on which $M$ is locally c-projectively flat---that is, locally isomorphic to ${\bf CP}^n$
equipped with its standard c-projective structure.
\end{theorem}

We obtain the analogous result for almost quaternionic structures in Theorem \ref{thm.almost.quat} below.


Proposition \ref{C(Z)_max_prop}  treats strongly essential flows at the other extreme from isolated zeros, rather with maximal \emph{strongly fixed sets} (see Definition \ref{defin_isol_zero} below).  For these, we show that the curves in $\mathcal T(\alpha)$, along which the curvature vanishes by Theorem \ref{thm.harmonic.vanishing}, always fill up a nonempty open set.  We thus obtain new rigidity results for almost Grassmannian, almost Lagrangean and almost spinorial structures (see 
Corollary \ref{Cor_C(Z)_max}); combining Proposition \ref{Case_isolated_fixed_point} and Corollary \ref{Cor_C(Z)_max} we obtain for example:

\begin{theorem}  
\label{thm.almost.grassmannian}
Let $M$ be endowed with a smooth $(2,n)$-almost Grassmannian structure, $n \geq 2$.  Suppose that $0 \neq \eta \in \mathfrak{inf}(M)$ has a higher-order zero 
at $x_0 \in M$.  Then there is an open subset $U \subset M$ with $x_0 \in \overline{U}$ on which $M$ is locally flat---that is, locally equivalent to 
the Grassmannian variety $\mathrm{Gr}(2,n+2)$. 
\end{theorem}

\subsection{Structure of the article}

Section \ref{sec.def.note} provides some background on irreducible parabolic geometries and summarizes the relevant results of \cite{cap.me.parabolictrans} on strongly essential flows.  The main result of Section \ref{sec.polynomial} is Proposition \ref{prop.s.to.the.k}, which estabishes a polynomial form for the harmonic curvature along the curves in $\mathcal{T}(\alpha)$.  Theorem \ref{thm.harmonic.vanishing} is proved in Section \ref{sec.main.proof}; the proof uses Proposition \ref{prop.s.to.the.k} and an analysis of the decomposition of $\lieg$ into irreducible components under the action of the $\algsl_2$-subalgebras of $\lieg$ determined by $T(\alpha)$.  
Section \ref{sec.applications} starts with a presentation of all the significant examples of irreducible parabolic geometries.  We then prove Propositions \ref{Case_isolated_fixed_point} and \ref{C(Z)_max_prop}, which lead to new rigidity results for a variety of geometries and strongly essential flows in Corollaries \ref{Cor_isolated_zero} and \ref{Cor_C(Z)_max}. Theorems \ref{thm.almost.cproj} and \ref{thm.almost.quat} are also proved in this section.
The paper concludes with a counter example to our rigidity results for higher-graded parabolic geometries, due to Kruglikov and The, and with the statement of some open questions.
\\\\
{\bf{Acknowledgments}}
\\We would like to thank Andreas \v Cap, Andre Chatzistamatiou, Michael Eastwood, Boris Kruglikov, Colleen Robles, and Dennis The for helpful discussions and comments.  Melnick was partially supported during work on this project by a
Centennial Fellowship from the American Mathematical Society and by NSF grants DMS-1007136 and 1255462.
 We thank the Mathematical Sciences Institute (MSI) at ANU for financially supporting with an MSRVP grant the visit of the first author to MSI in 2014. 

\section{Higher-order zeroes of irreducible parabolic geometries}
\label{sec.def.note}

We first briefly review some background on parabolic Cartan geometries. 
The reader is referred to \cite{sharpe} for the definition and basic examples of Cartan geometries.  This section briefly presents material on parabolic Cartan geometries as it will be used below.  The comprehensive reference on parabolic geometries is \cite{cap.slovak.book.vol1}.
Section \ref{sec.strongly.ess.flows} recalls some background on higher-order zeroes 
of infinitesimal automorphisms of parabolic geometries and the techniques from \cite{cap.me.parabolictrans} to study 
the local geometry around higher-order zeros.

\subsection{Parabolic geometries}\label{parabolic_geometries}
A convenient definition of parabolic subalgebras in semisimple Lie algebras is based on \cite[Lem 4.2]{Grothendieck.parab.subalgebra}. It reads as follows and is stated in this form in \cite{calderbank.diemer.soucek.ricci}:
\begin{defin}
Suppose $\lieg$ is a real or complex semisimple Lie algebra. A subalgebra $\liep \subset \lieg$ is called a \emph{parabolic subalgebra}
if the orthogonal complement $\liep^{\perp}$ of $\liep$ in $\lieg$ with respect to the Killing form coincides with the nilradical $\liep_+$ of 
$\liep$.
\end{defin}

It follows that the quotient $\liep/ \liep^{\perp}=\liep/ \liep_+$ is reductive, which is called
the \emph{Levi factor} of $\liep$ and is denoted $\lieg_0$, and that the Killing form induces an isomorphism 
$\liep_+\cong(\lieg/\liep)^*$ of $\liep$-modules. We will use this isomorphism without further mention 
to identify these two $\liep$-modules. Note that if $\lieg$ has nontrivial parabolic subalgebras, it is necessarily of non-compact type.


Let $k\geq 1$ be the degree of nilpotency of the nilpotent Lie algebra $\liep_+$. The 
lower central series of $\liep_+$ then equips $\lieg$ with the structure of a filtered Lie algebra
\begin{equation}\label{filtration}
\lieg=\lieg ^{-k}\supset \cdots \supset \lieg^{-1}\supset \lieg^0\supset \lieg^{1}\supset \cdots \supset\lieg^k \quad\quad [\lieg^i, \lieg^j] \subseteq \lieg^{i+j} \ \forall \ i,j\in\bold{Z},
\end{equation}
where $\lieg^1=\liep_+$, $\lieg^i=[\lieg^{i-1}, \liep_+]$ for $i\geq 2$ and $\lieg^{-j+1}=(\lieg^{j})^\perp$ for $j\geq 1$. Note that $\lieg^0=\liep$, whence
the filtration \eqref{filtration} is in particular $\liep$-invariant. If $\liep_+$ is abelian, the filtration takes the simple form 
$\lieg=\lieg^{-1}\supset \lieg^0\supset \lieg^{1}$ with $\lieg^1=\liep_+$.

The associated graded Lie algebra of the filtered Lie algebra \eqref{filtration} is a \emph{$|k|$-graded Lie algebra} 
\begin{equation}\label{associated_graded}
\textrm{gr}(\lieg)=\lieg_{-k}\oplus \cdots \oplus \lieg_0\oplus \cdots \oplus\lieg_k \quad\quad [\lieg_i, \lieg_j] \subseteq \lieg_{i+j} \ \forall \ i,j\in\bold{Z}
\end{equation}
where $\lieg_i= \lieg^{i}/\lieg^{i+1}$, and $\lieg_{-1}$ generates the subalgebra $\lieg_-=\oplus_{i\geq 1}\lieg_{-i}$.

It is easy to see that there exists a unique element $E_0\in\textrm{gr}(\lieg)$, called the \emph{grading element} of $\textrm{gr}(\lieg)$, 
such that $\textrm{ad}(E)$ acts by multiplication by $i$ on $\lieg_{i}$ for $-k\leq i\leq k$ (see \cite{cap.slovak.book.vol1}). Note that $E_0$ must 
lie in the center $\mathfrak z(\lieg_0)$ of $\lieg_0$. 

\begin{remark}
The filtration \eqref{filtration} is split, and a choice of such a splitting gives an identification of $\lieg$ with $\textrm{gr}(\lieg)$. There is however no canonical splitting. 
In \cite{calderbank.diemer.soucek.ricci} splittings of \eqref{filtration}
are called \emph{algebraic Weyl structures}. 
Without further mention,  we assume in this article that for any parabolic pair $(\lieg,\liep)$, an algebraic Weyl structure is fixed, so we have fixed 
an identification $\lieg \cong \textrm{gr}(\lieg)$. The results of this article are clearly independent of such a choice.
\end{remark}

Suppose $G$ is a real or complex semisimple Lie group with Lie algebra $\lieg$, and let $\liep$ be a parabolic subalgebra of $\lieg$.
Then any subgroup $P< G$ with Lie algebra $\liep$ is isogeneous to the stabilizer in $G$ of the filtration \eqref{filtration} under 
the adjoint representation $\textrm{Ad}$ of $G$ and is called a \emph{parabolic subgroup}.
We write $P_+=\exp(\liep_+)$ for the Ad-unipotent radical of $P$, and $G_0=P/P_+$ for the Levi factor. 
Having fixed a splitting of \eqref{filtration}, we can identify $G_0$ with the subgroup of $P$ preserving the grading on $\lieg\cong\textrm{gr}(\lieg)$.

\begin{defin} Let $P$ be a parabolic subgroup of a real or complex semisimple Lie group $G$. 
A \emph{parabolic geometry} of type $(\lieg,P)$ on a manifold $M$  is a real, smooth Cartan geometry $(B\stackrel{\pi}{\rightarrow}M,\omega)$ of type 
$(\lieg,P)$ on $M$.
\end{defin}

The homogeneous space $G/P$ equipped with the Maurer--Cartan form $\omega_G$ of $G$ is called the \emph{homogeneous model} of 
Cartan geometries of type $(\lieg, P)$. Assuming $G$ is an algebraic group, $P$ is a cocompact algebraic subgroup of $G$, so that $G/P$ is 
a closed projective variety. 



Representations of $P$ such that $P_+$ acts trivially are in bijective correspondence with representations of $G_0=P/P_+$.
A representation $\mathbb V$ of $P$ is completely reducible if and only if $P_+$ acts trivially on $\mathbb V$ and $\mathbb V$ is completely reducible as $G_0$-module.  Since $G_0$ is reductive, this last condition holds if and only if the center of $G_0$ acts by a character on $\mathbb V$. We will often identify sections of an associated vector bundle $V=B\times_P\Bbb V$ with smooth \emph{$P$-equivariant functions} $f: B\rightarrow \Bbb V$---that is, $f(bp)=p^{-1}f(b)$ for any $b\in B$ and $p\in P$. 

This article deals mainly with \emph{irreducible parabolic geometries}, defined by the property that the unipotent nilradical $P_+$ of $P$ is abelian. 
Such parabolic geometries are in the literature also called \emph{abelian parabolic geometries}, \emph{$|1|$-graded parabolic geometries}, or \emph{almost hermitian symmetric structures}.  

The Cartan connection of an irreducible parabolic geometry of type $(\lieg, P)$ induces a morphism between the $G_0$-principal bundle $B_0=B/P_+$ and the frame bundle 
of $M$, corresponding to the group homomorphism $G_0\rightarrow GL(\lieg/\liep)$. Hence, the Cartan connection induces a first-order $G_0$-structure on $M$. 
Under some homological condition on the pair $(\lieg,\liep)$, the prolongation procedures of \cite{tanaka.parabolic}, \cite{morimoto} and \cite{cap.schichl.equiv}, 
associate to a first-order $G_0$-structure on $M$ a canonical Cartan connection of type $(\lieg, P)$, called the \emph{normal Cartan connection}. 
There is thus an equivalence of categories between 
normal irreducible parabolic geometries of type $(\lieg,P)$ and first order $G_0$-structures on $M$.  For projective and almost c-projective structures, this homological condition is not satisfied, 
but these structures nonetheless determine a canonical irreducible parabolic geometry (see \cite[3.1.16 ]{cap.slovak.book.vol1}).

Thanks to the categorical equivalence above, infinitesimal automorphisms $\mathfrak{inf}(M)$ of underlying geometric structures on $M$ lift to infinitesimal automorphisms $\mathfrak{inf}(B,\omega)$ of the associated Cartan geometry.
The latter are $P$-invariant vector fields $\tilde{\eta} \in \mathfrak{X}(B)$ satisfying $\mathcal L_{\tilde\eta}\omega=0$. 
The flow $\varphi^t_{\tilde{\eta}}$, where defined, acts by automorphisms of the Cartan geometry, which are defined analogously.

Let us now explain the notion of normality of a parabolic geometry. 
The curvature of the Cartan connection $K \in \Omega^2(B,\lieg)$ can be identified via $\omega$ with a $P$-equivariant function
$$\kappa: B \rightarrow \Lambda^2 (\lieg/\liep)^* \otimes \lieg\cong \Lambda^2\liep_+\otimes \lieg.$$
The $P$-module $\Lambda^2\liep_+\otimes \lieg$ belongs to a complex of $P$-modules $\Lambda^* \liep_+ \otimes \lieg$, computing the 
Lie algebra homology of $\liep_+$ with coefficients in $\lieg$. 
A parabolic geometry is \emph{normal} if $\kappa$ has values in the kernel of $\partial^*$. 
The quotient $\Ker\partial^*/ \im\partial^* \cong H_*(\liep_+,\lieg)$ is a completely reducible representation of $P$ (see e.g.\,\cite{cap.slovak.book.vol1}), which therefore factors through $G_0$.  

The projection of the curvature $\kappa$ of a normal parabolic geometry to the quotient $\Ker \partial^* / \im \partial^*$ is called
the \emph{harmonic curvature}.  Having identified $G_0$ with a subgroup of $P$, Kostant's description of $H_*(\liep_+,\lieg)$ in 
\cite{kostant.harmonic.curv} yields a natural identification of the $G_0$-module $H_2(\liep_+,\lieg)$ with a $G_0$-submodule 
$\widehat {\mathbb W}$ of $\mathbb W=\Lambda^2\liep_+\otimes \lieg$ such that as $G_0$-modules,
\begin{equation}\label{ker_partial*}
\Ker \partial^*=\widehat{\mathbb W}\oplus \im \partial^*.
\end{equation} Hence,
we may view the harmonic curvature as a $G_0$-equivariant function
$$\hat\kappa: B\rightarrow \widehat{\mathbb W}\subset \Ker \partial^*\subset \mathbb W,$$
which is constant along the fibers of $B\rightarrow B_0=B/P_+$. 

For normal, irreducible parabolic geometries, the full curvature is related to the harmonic curvature via a differential operator $S$---that is, $S(\hat{\kappa}) = \kappa$.  Therefore vanishing of $\hat{\kappa}$ over on open set $U \subset M$ implies vanishing of $\kappa$ over $U$ (see \cite[Theorem 3.1.12]{cap.slovak.book.vol1}, \cite{calderbank.diemer}).


\subsection{Strongly essential infinitesimal automorphisms}
\label{sec.strongly.ess.flows}
We recall now some results of \cite{cap.me.parabolictrans}, in the case of irreducible parabolic geometries.  

The higher-order zeroes of strongly essential automorphisms can be classified by their geometric types:

\begin{defin}\label{geometric.type} Suppose $(B\stackrel{\pi}{\rightarrow}M,\omega)$ is a normal irreducible parabolic geometry. Assume $\eta\in\mathfrak{inf}(M)$ is strongly essential with a higher-order zero at $x_0\in M$.  Let $Z$ be the isotropy of $\eta$ with respect to $b_0 \in \pi^{-1}(x_0)$, with corresponding $\alpha\in T_{x_0}^*M$. The $P$-orbit of $Z$ in $\liep_+$ is called the \emph{geometric type} of $\alpha$. 
\end{defin}
The geometric type is clearly independent of the choice of $b_0\in\pi^{-1}(x_0)$.  For irreducible geometries, $\liep_+$ is a completely reducible $P$-module, so the $P$-orbits coincide with the 
$G_0$-orbits. 


On a parabolic homogeneous model $(G \rightarrow G/P, \omega_G)$ the left action of a 1-parameter subgroup $e^{tZ}$, $Z \in \liep_+$, is a strongly essential flow.  The isotropy of this flow at $x_0 = 1_GP$ with respect to $1_G$ is $Z$.  The methods of \cite{cap.me.parabolictrans} lie in a comparison of strongly essential flows with their corresponding isotropy flows on the homogeneous model.

Any $X \in \lieg$ defines a vector field $\widetilde{X}$ on $B$ by $\omega(\tilde{X}) \equiv X$.  The exponential map of a Cartan geometry is given 
by the time-one flow along these $\omega$-constant vector fields.
\begin{defin}
\label{def.exp.map}
Let $(B \stackrel{\pi}{\rightarrow} M, \omega)$ be a Cartan geometry.  The \emph{exponential map} at $b \in B$ is 
$$ \exp_b(X) = \varphi^1_{\widetilde{X}}(b) \in B$$
for $X$ in a sufficiently small neighborhood of $0$ in $\lieg$.
\end{defin}
The restriction of $\exp_b$ to a sufficiently small neighborhood of $0$ in $\lieg$ is a 
diffeomorphism onto a neighborhood of $b$ in $B$. The map $\pi\circ \exp_b$ induces a diffeomorphism from a neighborhood of $0$ in $\lieg_-$ to a neighborhood of $x=\pi(b)$. Projections to $M$ of exponential curves 
$s\mapsto\exp(b,sX)$ for $X\in\lieg_-$ are called \emph{distinguished curves}. 

Suppose now that $\eta\in\mathfrak{inf}(M)$ has a 
higher-order zero at $x_0$ and that the isotropy of $\eta$ with respect to $b_0\in\pi^{-1}(x_0)$ is $Z\in\liep_+$. Assume that for any fixed $t \in \BR$ and $X \in \lieg$, the following equation holds in $G$:
\begin{equation}\label{isotropy_flow}
e^{tZ}e^{sX}=e^{c_t(s)X}p_t(s) \qquad \forall s \in I
\end{equation} 
Here $I$ is an interval containing $0$, and $c_t: I \stackrel{\sim}{\rightarrow} I'$ is a diffeomorphism fixing $0$; $p_t: I\rightarrow P$ is a smooth path with $p_t(0)=e^{tZ}$.  Hence, the flow $e^{tZ}$ acts on the curve $e^{sX}P$ in $G/P$ by a reparametrization. 
Then it follows from \cite[ Prop 4.3]{fm.nilpconf} or \cite[Prop 2.1]{cap.me.parabolictrans} that the analogous equation holds in $B$:
\begin{equation}\label{action_flow_in_B}
\varphi^t_{\tilde\eta}\exp(b_0, sX)=\exp(b_0, c_t(s)X)p_t(s)\quad \forall s\in I.
\end{equation} 

\begin{defin}
The \emph{commutant} of $Z \in \liep_+$ is
$$C(Z)=\{X\in\lieg_- \ : \ [Z,X]=0\}$$
\end{defin}

It follows from equation \eqref{isotropy_flow} that for $X\in C(Z)$, the curve $\gamma_X$ consists of higher-order fixed points of $\varphi^t_\eta$ of the same geometric type as $x_0$; in fact, the isotropy of $\eta$ at $\gamma_X(s)$ with respect to $\exp(b_0,sX)$ equals $Z$ (see \cite[Prop 2.5]{cap.me.parabolictrans}). 
As with $T(\alpha)$, the commutant with respect to $b_0$ determines a well-defined subset $C(\alpha) \subset T_{x_0}M$, independent of $b_0\in\pi^{-1}(x_0)$.  

\begin{defin}\label{defin_isol_zero} 
Suppose $\eta\in\mathfrak{inf}(M)$ has a higher-order zero at $x_0\in M$.
\begin{itemize}
\item The \emph{strongly fixed component} of $x_0$ in a neighbhorhood $U$ is the set of all endpoints of smooth curves in $U$ emanating from $x_0$ consisting of higher-order zeroes of $\eta$ of the same geometric type as $x_0$. 
\item
The higher-order zero $x_0$ is called \emph{smoothly isolated} if it equals its strongly fixed component in some neighborhood. 
\end{itemize}
\end{defin}
By \cite[Prop 2.5]{cap.me.parabolictrans}, a higher-order zero with isotropy $\alpha$ is smoothly isolated if and only if $C(\alpha)=\{0\}$.


\begin{definition} Suppose $\varphi^t_{\tilde\eta}$ is a flow by automorphisms of a parabolic geometry $(B \stackrel{\pi}{\rightarrow} M, \omega)$.
A \emph{holonomy path} for $\{\varphi^t_{ \tilde\eta} \}$ at $b \in B$ with \emph{attractor} $b_0$ is a path $p_t \in P, t \geq 0$, so that 
for some path $b_t \stackrel{t \rightarrow \infty}{\longrightarrow} b$, 
$$\varphi^t_{\tilde\eta\,} b_t p_t^{-1} \rightarrow b_0\qquad \textrm{ as } t\rightarrow\infty .$$
\end{definition}

Suppose $\mathbb V$ is a representation of $P$, and $\tau$ 
a $P$-equivariant, $\varphi^t_{\tilde{\eta}}$-invariant function $B\rightarrow \Bbb V$. 
If $p_t$ is a holonomy path at $b$ with attractor $b_0$, then
$$\tau(\varphi^t_{\tilde{\eta}} b_t p_t^{-1}) = p_t \cdot \tau(b_t) \rightarrow \tau(b_0) \quad \textrm{ as } t \rightarrow \infty.$$
Since $\tau(b_t) \stackrel{t \rightarrow \infty}{\longrightarrow} \tau(b)$, it is clear that $\tau(b)$ must be bounded under 
$p_t$ as $t \rightarrow \infty$. Hence, a holonomy path at $b$ gives restrictions on the values of an invariant section of the associated bundle $B \times_P \Bbb{V}$ at $\pi(b)$.

Given an irreducible normal parabolic geometry, assume $\eta\in\mathfrak{inf}(M)$ has a higher-order zero at $x_0\in M$,
and let $Z\in\liep_+$ be the isotropy of $\eta$ with respect to $b_0\in\pi^{-1}(x_0)$. For any $X \in T(Z)\subset\lieg_{-1}$ as defined in Section \ref{sec.intro.results}, $A=[Z,X]\in\lieg_0$ and, by \cite[Prop 2.12]{cap.me.parabolictrans}, 
the action of $e^{tZ}$ on $e^{sX}$ in $G$ satisfies
$$e^{tZ}e^{sX}=e^{\frac{s}{1+st}X}a_t(s)u_t(s)\quad\quad s,t\in \BR, \ st > 0$$
where $a_t(s)=e^{\log(1+st)A}\in G_0$ and $u_t(s)=e^{\frac{t}{1+st} Z}\in P_+$.
(See also Section \ref{sec.cproj.improvement} below, where this formula is derived in a slightly more general setting.)  Hence \eqref{action_flow_in_B} implies the following key equation for $X \in T(Z)$:
\begin{eqnarray}
\label{eqn.sl2.holo}
\varphi^t_{\tilde\eta} \tilde{\gamma}_X (s) = \tilde{\gamma}_X \left( \frac{s}{1+st} \right) a_t(s) u_t(s) \quad \forall s\in I, \ st > 0
\end{eqnarray}
where $I$ is the domain of  ${\tilde\gamma}_X$.
Thus for $s > 0$, $p_t(s) = a_t(s) u_t(s)\in P$ is a holonomy path for $\varphi^t_{\tilde\eta}$ at $\tilde{\gamma}_X(s)$ with attractor $b_0$; for $s <0$, the same holds for $\varphi^{-t}_{\tilde{\eta}}$ (see \cite[Prop 2.12]{cap.me.parabolictrans}).  Alternatively, for $s > 0$, 
$a_t(s)\in G_0$ is a holonomy path for $\varphi^t_{\tilde\eta}$ at $\tilde{\gamma}_X(s) e^{- \frac{1}{s} Z}$ with attractor $b_0$ (respectively for $s <0$, $\varphi^{-t}_{\tilde{\eta}}$).  We will use both forms of the holonomy path.

Because, as above, $A \in \lieg_0$, representation theory of $\algsl(2)$ gives that $A$ acts diagonalizably with integer eigenvalues on any completely reducible representation of $\lieg_0$.  The holonomy path $a_t(s)$ therefore acts diagonalizably with eigenvalues $(1+st)^\ell$ on any completely reducible $G_0$-module $\mathbb V$, where $\ell$ is an eigenvalue of $A$ on $\mathbb V$.
Write $\Bbb V^{[\ell]}(A)$ for the $\ell$-eigenspace of $A$ on $\mathbb{V}$.
\begin{defin}
The \emph{stable subspace} and \emph{strongly stable subspace} of $A$ on $\Bbb{V}$ are, respectively,
$$\Bbb V^{[st]}(A) = \oplus_{\ell \leq 0}\Bbb V^{[\ell]} \qquad \mbox{and} \qquad \Bbb V^{[ss]}(A) = \oplus_{\ell < 0} \Bbb V^{[\ell]}$$
\end{defin}
If $A$ is understood, we will also simply write $\Bbb V^{[st]}$ or $\Bbb V^{[ss]}$.

Now assume $\Bbb V$ is a completely reducible $P$-module, so $P_+$ acts trivially on it.
Let $\tau: B\rightarrow \Bbb V$ be a $P$-equivariant function that is invariant under a strongly essential flow $\varphi^t_{\tilde\eta}$. Then 
we deduce from equation \eqref{eqn.sl2.holo} (see \cite[Prop 2.15]{cap.me.parabolictrans}) that for any $X\in T(Z)$, 
\begin{align}\label{tau_along_special_curves}
&\bullet\, \tau(\tilde{\gamma}_X(s))\in\Bbb V^{[st]}(A) \quad \textrm{ for all } s\in I\\ 
&\bullet\, \tau(\tilde{\gamma}_X(s))\in\Bbb V^{[ss]}(A)\quad \textrm{ for all } s\in I \textrm{ provided } \tau(b_0)=0\nonumber.
\end{align}


In \cite{cap.me.parabolictrans} the authors apply \eqref{tau_along_special_curves} to the harmonic curvature $\hat\kappa: B\rightarrow \widehat{\Bbb W}$.
They verify for various parabolic geometries admitting strongly essential flows that $\widehat {\Bbb W}^{[st]}(A)=\{0\}$, or that 
$\hat\kappa(b_0)=0$ and $\widehat {\Bbb W}^{[ss]}(A)=\{0\}$, for any $X\in T(Z)$, which implies vanishing of $\hat\kappa$
along the curves $\gamma_X$ for $X\in T(Z)$. In many cases however $\widehat{\Bbb W}^{[ss]}(A)$ for $X\in T(Z)$ is nontrivial and so stronger techniques 
are required to obtain vanishing of the harmonic curvature along the curves $\gamma_X$. The main result of this article, Theorem \ref{thm.harmonic.vanishing}, 
shows that for normal irreducible parabolic geometries the harmonic curvature always vanishes along the curves $\gamma_X$ emanating from the fixed point of a strongly essential flow.

\section{Vanishing of the harmonic curvature along special curves}
\label{sec.main.proof}

This section contains the proof of Theorem \ref{thm.harmonic.vanishing}.  The first step of the proof is Proposition \ref{prop.s.to.the.k}, which establishes that $\hat\kappa(\tilde{\gamma}_X(s))$ is a polynomial in $s$ with respect to a parallel moving frame along the curve $\tilde{\gamma}_X$.  We then use Proposition \ref{prop.s.to.the.k} to show that for normal irreducible parabolic geometries 
of type $(\lieg, P)$ with $\lieg$ simple, existence of an infinitesimal automorphism with a higher-order zero implies vanishing 
of the harmonic curvature along the curves in $\mathcal T(\alpha)$. 
The proof of Theorem \ref{thm.harmonic.vanishing} also shows that the full Cartan curvature vanishes at the zero. A unified proof of 
Theorem \ref{thm.harmonic.vanishing} for all irreducible parabolic geometries requires first 
a thorough study of $\algsl(2)$-triples $\{X, A, Z\}$ of simple $|1|$-graded Lie algebras $\lieg$ 
that are adapted to the grading in such a way that $X\in\lieg_{-1}$, $A\in\lieg _0$ and $Z\in\lieg_1$.
Using the representation theory of $\algsl(2)$, we prove general facts about the 
eigenvalues of the action of semisimple elements $A\in\lieg_0$ 
of adapted $\algsl(2)$-triples on the grading components 
of $\lieg$.

\subsection{Polynomial expression of invariant sections along special curves}
\label{sec.polynomial}
In this section $(B\stackrel{\pi}{\rightarrow}M,\omega)$ is an arbitrary parabolic geometry admitting a nontrivial infinitesimal automorphism $\tilde{\eta} \in \mathfrak{inf}(B,\omega)$.  We suppose that $\omega(\tilde\eta(b_0)) = Z \in \liep_+$ for some $b_0\in B$. 


\begin{proposition}\label{prop.s.to.the.k} 
Suppose $V=B\times_P\mathbb V$ is a completely reducible vector bundle and suppose $\tau: B\rightarrow \Bbb V$ corresponds to a 
$\{ \varphi^t_{\tilde\eta} \}$-invariant section of $V$. Then for any $X \in T(Z)$ such that $A=[Z,X]\in\lieg_0$ and any integer $k \geq 0$, the component of 
$\tau$ in $\mathbb{V}^{[-k]}(A)$ satisfies
$$ (\tau \circ \tilde{\gamma}_X)^{[-k]}(s) = s^k v^{[-k]}\quad\textrm{ for some }\, v^{[-k]} \in\Bbb{V}^{[-k]}(A).$$
Therefore, for all $s$ in the domain $I$ of $\tilde{\gamma}_X$, 
$$(\tau\circ \tilde{\gamma}_X)(s) = \sum_{k \geq 0} s^{k} v^{[-k]}\quad \textrm{ where }\, v^{[-k]} \in \Bbb{V}^{[-k]}(A).$$
\end{proposition}

\begin{proof}  Fix $X \in T(Z)$ and write $\tilde{\gamma} = \tilde{\gamma}_X$.  By \eqref{tau_along_special_curves} the curve $\tau\circ\tilde\gamma$ 
has values in $\Bbb V^{[st]}=\Bbb V^{[st]}(A)$ for $A = [Z,X]$. Let $k \geq 0$ be an integer and $e_1, 
\ldots, e_N$ be a basis of $\mathbb{V}^{[-k]}=\mathbb{V}^{[-k]}(A)$.  The goal is to show that for some constants $c^i$, 
$$ (\tau \circ \tilde{\gamma})^{[-k]}(s) = s^k \sum_i c^i e_i$$
Begin by writing
$$ (\tau \circ \tilde{\gamma})^{[-k]}(s) = \sum_i f^i(s) e_i.$$

Complete reducibility of $\Bbb{V}$ and the invariance of $\tau$ gives, by formula (\ref{eqn.sl2.holo}),
\begin{eqnarray*}
(\tau \circ \tilde{\gamma})^{[-k]}(s) & = & \tau^{[-k]} (\varphi^t \tilde{\gamma}(s))\\
 &=&  \tau^{[-k]} (\tilde{\gamma} \left( \frac{s}{1+st} \right) a_s(t) e^{\frac{t}{1+st} Z} ) \\
 & = & a_s(t)^{-1} \cdot (\tau \circ \tilde{\gamma} )^{[-k]} \left( \frac{s}{1+st} \right) \\
 & = & a_s(t)^{-1} \cdot \sum f^i \left( \frac{s}{1+st} \right)  e_i \\
 & = & (1+st)^{k} \sum f^i \left( \frac{s}{1+st} \right)  e_i
 \end{eqnarray*}
 
 because $e_i \in \mathbb{V}^{[-k]}$.  Thus
 \begin{eqnarray}
 \label{eqn.f.equivariance}
  f^i \left( \frac{s}{1+st} \right) = (1+st)^{-k} f^i(s)
  \end{eqnarray} 
 
Fix $0 \leq s_0 \in I$.  For $t \in [0,\infty)$, the variable $x = \frac{s_0}{1+s_0t}$ ranges from $s_0$ to $0$.  Equation (\ref{eqn.f.equivariance}) reads
$$f^i(x) = \frac{x^k}{s_o^k} f^i(s_0)$$ 
Thus for $s \in [0,s_0]$, $f^i(s) = C_+ s^k$, where $C_+ = f^i(s_0)/s_0^k$.  The same holds for negative $s \in I$, for some $C_-$, when $t$ ranges from $0$ to $- \infty$.  Now note that 
$$ k! C_+ = \lim_{s \rightarrow 0^+} \frac{\textrm{d}^k}{\textrm{d}s^k} f^i(s) = \lim_{s \rightarrow 0^-} \frac{\textrm{d}^k}{\textrm{d}s^k} f^i(s) = k! C_-$$
so $C_+ = C_-$.  
\end{proof}  

\begin{remark}
It is possible to prove a discrete-time version of this proposition.  More precisely, if there is a sequence of automorphisms $\{h_k \}$ of $(M,B,\omega)$ fixing $x_0$ and a curve $\tilde{\gamma}_X$ through $b_0 \in \pi^{-1}(x_0)$ for which 
$$ h_k \tilde{\gamma}_X(s) = \tilde{\gamma}_X \left(\frac{s}{1+st_k} \right) a_s(t_k) e^{\frac{t_k}{1+st_k}Z} $$
then the conclusion of Proposition \ref{prop.s.to.the.k} holds.  The proof can be found in a previous version of this paper at \cite{me.neusser.parabolic.arxiv}.
\end{remark}

 \subsection{Recall of representation theory of  $\algsl(2)$}
 \label{Recall_sl2}

Let us briefly recall the representation theory of $\algsl(2,\bold{K})$ for $\bold{K}=\BC,\BR$. 
It is well known that for any integer $\ell\geq 0$ 
there exists an irreducible $\algsl(2,\bold{K})$-module $\mathbb V(\ell)$ 
of $\bold{K}$-dimension $\ell+1$, unique up to isomorphism.  Denote  
the standard generators of $\algsl(2,\bold{K})$ by
\begin{equation}\label{standard_generators}
X= \left(\begin{array}{cc}
0& 0\\ 
1& 0\\
\end{array}
\right),\,\,
Z= \left(\begin{array}{cc}
0& 1\\ 
0& 0\\
\end{array}
\right)\, \textrm{ and }\,\,
A=[Z,X]= \left(\begin{array}{cc}
1& 0\\ 
0& -1\\
\end{array}
\right).
\end{equation} 

With respect to the action of the semisimple element $A$ 
the module $\mathbb V(\ell)$ decomposes into a direct sum of $1$-dimensional weight spaces with weights 
$\ell, \ell-2,\ldots,-\ell+2,-\ell$. In particular, $\ell$ is the highest weight of $\mathbb V(\ell)$.
Let $v_0$ be a highest weight vector of $\mathbb V(\ell)$ and set $v_j=\frac{1}{j!}X^j v_0$ for $j=1,\ldots,\ell$. Then
the elements $\{v_0, v_1, v_2,\ldots,v_\ell\}$ form a basis of $\mathbb V(\ell)$ that consists of weight vectors of $A$. 
The action of the standard generators on these basis elements is given by
\begin{equation}\label{standard_generators2}
 Av_j=(\ell-2j)v_j\qquad
Xv_j=(j+1)v_{j+1}\qquad
 Zv_j=(\ell-j+1)v_{j-1},
\end{equation} 
where we decree $v_{-1}=0$ and $v_{\ell+1}=0$. In particular, this implies that $v_0=\frac{1}{\ell!}Z^{\ell}v_{\ell}$.



\subsection{Adapted $\algsl(2)$-triples in $|1|$-graded semisimple Lie algebras}
\label{adapted_sl2_triples}
Suppose $\lieg$ is a complex or real semisimple Lie algebra. Let $0\neq Z\in\lieg$ be a nilpotent element, meaning $\textrm{ad}(Z)$ 
is nilpotent on $\lieg$. We have already remarked that $Z$ can be completed 
to an $\algsl(2)$-triple $\{X, A, Z\}$ of the following form:
\begin{equation}\label{sl2}
A=[Z,X]\qquad [A,Z]=2Z\qquad [A,X]=-2X.
\end{equation}
This implies in particular that $Z$ has degree of nilpotency $\geq 3$ in $\lieg$. Conversely, by elementary representation theory 
(see also Proposition \ref{eigenvalues} below), for any $\algsl(2)$-triple as in \eqref{sl2}, the element $Z\in\lieg$ is a nilpotent element 
in $\lieg$ of degree of nilpotency $\geq 3$. The element $A$ acts 
diagonally on $\lieg$ with integer eigenvalues; the $j$-eigenspace is denoted $\lieg^{[j]}$ (cf. Section \ref{sec.strongly.ess.flows}). 
We write $\lieg^{[j]}$ even if $j\in\bold Z$ is not an eigenvalue of $A$, in which case this space is $\{0 \}$. 

\begin{proposition}\label{eigenvalues}  
Suppose $\lieg$ is a complex or real semisimple Lie algebra and let $\{X, A, Z\}$ be an $\algsl(2)$-triple in $\lieg$ as in \eqref{sl2}. 
If $Z$ has degree of nilpotency $3$, then $A$ acts diagonally on $\lieg$ with integer eigenvalues between $-2$ and $2$. The operators $\emph{ad}(Z)$ and $\emph{ad}(X)$induce inverse 
isomorphisms between $\lieg^{[-1]}$ and $\lieg^{[1]}$, and $\frac{1}{2}\emph{ad}^2(Z)$ and $\frac{1}{2}\emph{ad}^2(X)$ induce inverse isomorphisms between $\lieg^{[-2]}$ and $\lieg^{[2]}$.
\end{proposition}

\begin{proof}
Suppose $\lieg$ is a complex or real semisimple Lie algebra, let $\{X, A, Z\}$ be an $\algsl(2)$-triple in $\lieg$, and 
denote by $\mathfrak a$ the complex or real $\algsl(2)$ subalgebra generated by $X$, $A$, and $Z$. 
As an $\mathfrak a$-module, $\lieg$ decomposes into irreducible components.  From Section \ref{Recall_sl2}, each irreducible component is isomorphic to $\mathbb V(\ell)$ for some $\ell\geq 0$. Choose for each component 
an adapted basis of weight vectors as in \eqref{standard_generators2}.
It is immediately apparent that $Z$ is a nilpotent element in $\lieg$ of degree $\ell_0+1$ if and only if $\ell_0$ is the largest positive integer such that $\mathbb V(\ell_0)$ occurs as an irreducible submodule 
in $\lieg$. Hence, if $Z$ has degree of nilpotency $3$, then all irreducible components in $\lieg$ are isomorphic to either $\mathbb V(0)$, $\mathbb V(1)$, or $\mathbb V(2)$. Thus all eigenvalues of $A$ are integers between $-2$ and $2$.
Moreover, provided that $\lieg^{[1]}\neq \{0\}$, any nonzero elements in $\lieg^{[1]}$ and $\lieg^{[2]}$ are highest weight vectors of the $\mathfrak a$-module $\lieg$.  The second claim follows from \eqref{standard_generators2}.
\end{proof}

Suppose now that $\lieg$ is a complex or real semisimple Lie algebra equipped with a $|1|$-grading.
Let $Z\in\lieg_1$ be nonzero and fix an element $X\in T(Z)$. 
Since $\textrm{ad}^3(Z)=0$, the elements $X\in\lieg_{-1}$, $A=[Z,X]\in\lieg_0$ and $Z\in\lieg_1$ form 
an $\algsl(2)$-triple in $\lieg$ such that $Z$ has degree of nilpotency $3$. Hence, by Proposition
\ref{eigenvalues} the semisimple element $A$ acts diagonally on $\lieg$ with integer eigenvalues 
between $-2$ and $2$. Also $A \in \lieg_0$, so $A$ preserves each $\lieg_i$, with eigenspace decomposition
\begin{equation}\label{eigenspace_decomposition}
\lieg_i=\bigoplus_ {j=-2}^{2} \lieg_i^{[j]}\quad i=1,0-1.
\end{equation}
By the Jacobi identity, 
\begin{equation}\label{compatibility_grading_eigenspaces}
[ \lieg_i^{[j]}, \lieg_k^{[\ell]}]\subset  \lieg_{i+k}^{[j+\ell]} 
\end{equation}
for all $-1 \leq i,k \leq 1$ and 
$-2 \leq j,\ell \leq 2$, where we decree that $\lieg_{r}^{[s]}=\{0\}$ for $|r|>2$ or $|s|>3$.

As above, let $\mathfrak a$ be the $\algsl(2)$ subalgebra generated by $X$, $A$, and $Z$, and let $\mathfrak z(Z)$ be the centralizer of $Z$ in $\lieg$.
Note that $\mathfrak z(Z)$ is precisely the subspace of all the highest weight vectors of irreducible $\mathfrak a$-modules in $\lieg$. By the Jacobi identity, $\mathfrak z(Z)$ is 
an $A$-invariant subalgebra of $\lieg$. By Proposition \ref{eigenvalues} its decomposition into eigenspaces of $A$ is therefore given by
\begin{equation}\label{center_decomposition}
\mathfrak z(Z)=\lieg^{[2]}\oplus\lieg^{[1]}\oplus(\lieg^{[0]}\cap\mathfrak z(Z)). 
\end{equation}


Decomposing $\lieg$ into irreducible components 
for the action of $\mathfrak a$, Proposition \ref{eigenvalues} and \eqref{compatibility_grading_eigenspaces}
immediately give detailed information about the eigenspace decomposition of $A$ 
on the individual grading components of $\lieg$:

\begin{proposition} \label{eigenvalues_on_grading_components}
Suppose $\lieg=\lieg_{-1}\oplus\lieg_0\oplus\lieg_1$ is a complex or real $|1|$-graded semisimple Lie algebra. 
Let $Z\in\lieg_1$ be nonzero. Fix an element $X\in T(Z)$ and set $A=[Z,X]\in\lieg_0$. Then:
\begin{itemize}
\item[(a)] The linear operator $\emph{ad}(Z)$ induces isomorphisms $\lieg_{-1}^{[-1]}\cong\lieg_0^{[1]}$
and $\lieg_{0}^{[-1]}\cong\lieg_{1}^{[1]}$, the inverses of which are induced by $\emph{ad}(X)$.
\item[(b)] The linear operator $\emph{ad}^2(Z)$ induces an isomorphism $\lieg_{-1}^{[-2]}\cong\lieg_1^{[2]}$, and $\emph{ad}^2(X)$ induces an isomorphism the other way.
\item[(c)] The possible eigenvalues of $A$ on $\lieg_{-1}$ are $-2$, $-1$, and $0$; moreover, $\lieg_{-1}^{[0]}=C(Z)$. 
\item[(d)] The possible eigenvalues of $A$ on $\lieg_{1}$ are $2$, $1$, and $0$.
\item[(e)] The possible eigenvalues of $A$ on $\lieg_0$ are $-1$, $0$, and $1$; moreover, 
$$\lieg_0^{[1]}=[\lieg_{-1}^{[-1]},Z]=\mathfrak z(Z)\cap[\lieg_{-1},Z]\subset \lieg_0,$$ which is an abelian 
ideal of the subalgebra $\mathfrak z(Z)\cap\lieg_0$.  
\end{itemize}
\end{proposition}

\begin{proof}
Let $0 \neq Z \in \lieg_1$, fix $X\in T(Z)$ and set $A=[Z,X]\in\lieg_0$. 
Denote by $\mathfrak a$ the $\algsl(2)$ subalgebra of $\lieg$ generated by 
$X$, $A$, and $Z$. Since $\textrm{ad}^3(Z)=0$, we have already noticed that Proposition \ref{eigenvalues} applies
and hence that $\lieg$ decomposes as an $\mathfrak a$-module into a direct sum 
of irreducibles, each isomorphic to either $\mathbb V(0)$, $\mathbb V(1)$, or $\mathbb V(2)$. 
In particular, the eigenvalues of $A$ on $\lieg$ are integers between $-2$ and $2$.  The fact that 
$A\in\lieg_0$ implies that the decomposition into eigenspaces of $A$ of the individual grading components 
is given by \eqref{eigenspace_decomposition}.  From Proposition \ref{eigenvalues}, $\textrm{ad}(Z)$ and $\textrm{ad}(X)$, 
respectively $\frac{1}{2}\textrm{ad}(Z)$ and $\frac{1}{2}\textrm{ad}^2(X)$, induce inverse isomorphisms between $\lieg^{[-1]}$ and $\lieg^{[1]}$, respectively between $\lieg^{[-2]}$ 
and $\lieg^{[2]}$. Since $Z\in\lieg_1^{[2]}$ and $X\in\lieg_{-1}^{[-2]}$, we therefore 
conclude from \eqref{compatibility_grading_eigenspaces} 
that $(a)$, $(b)$, $(d)$ and the first statements of $(c)$ and $(e)$ hold 
(note that $(d)$ also follows immediately from the first statement of $(c)$, 
since $\lieg_{-1}^*\cong\lieg_1$ as $\lieg_0$-modules). 

To complete the proof of $(c)$ it remains to show that $\lieg_{-1}^{[0]}=C(Z)$. Suppose first that $W \in C(Z)$, so $[Z,W]=0$. Then the Jacobi identity and the fact that $\lieg_{-1}$ is an abelian subalgebra of $\lieg$ immediately 
imply that $[X, W]=0$ and $[A,W]=0$. Hence, $C(Z)\subset \lieg_{-1}^{[0]}$, and if $C(Z)\neq 0$, any choice of basis for $C(Z)$ identifies $C(Z)$ with a direct sum of copies of the trivial representation $\mathbb V(0)$ of $\mathfrak a$. 
Conversely, assume now that $W\in\lieg_{-1}^{[0]}$ does 
not generate a copy of the trivial representation, so $W \in \lieg_{-1}^{[0]}\setminus C(Z)$. Then $W$ has 
to be of the form $W=[X,V]$ for a highest weight vector $V\in\lieg^{[2]}$, which implies that $[X,W]=[X, [X,V]]\neq 0$. 
But this is impossible, since $\lieg_{-1}$ is abelian. Hence $C(Z)=\lieg_{-1}^{[0]}$.
 
Now let us prove the second assertion in $(e)$. Note first that $(a)$ implies that  $\lieg_0^{[1]}=[\lieg_{-1}^{[-1]}, Z]$. 
Since $\lieg_{-1}^{[0]}=C(Z)$ by $(c)$ and $\textrm{ad}^2(Z)(Y)\neq 0$ for any nonzero element $Y\in\lieg_{-1}^{[-2]}$ by $(b)$, 
we deduce that $\mathfrak z(Z)\cap[\lieg_{-1},Z]=\mathfrak z(Z)\cap[\lieg_{-1}^{[-1]},Z]$. The latter space in turn equals 
$[\lieg_{-1}^{[-1]}, Z]=\lieg_{0}^{[1]}$, since $\lieg_{0}^{[1]}$ consists of highest weight vectors.
Since $[\lieg_0^{[1]}, \lieg_0^{[1]}]\subset \lieg_{0}^{[2]}=\{0\}$ by \eqref{compatibility_grading_eigenspaces} 
and the first statement of $(e)$, the subspace $\lieg_{0}^{[1]}$ is 
an abelian subalgebra of $\lieg_0$.  Note $\mathfrak{z}(Z) \cap \lieg_0 \subset \lieg_0^{[0]} \oplus \lieg_0^{[1]}$ by \eqref{center_decomposition}; with the Jacobi identity, one sees that $\lieg_0^{[1]}$ is an abelian ideal in here, as claimed. 
\end{proof}

Suppose $0 \neq Z \in \lieg_1$ for $\lieg$ a $|1|$-graded semisimple Lie algebra, and let $X \in T(Z)$. From Proposition \ref{eigenvalues_on_grading_components}, the $1$-eigenspace $\lieg_0^{[1]}$ of $A=[Z,X]$ on $\lieg_0$ is an abelian subalgebra of $\lieg_0$. Hence, the exponential
group $G_{0}^{[1]}=\exp({\lieg_0^{[1]}})$ is an abelian subgroup of $G_0$, which acts on $\lieg_{-1}$ by restriction of 
the adjoint action. For any $U\in\lieg_{0}^{[1]}$ and any $Y\in\lieg_{-1}$,
\begin{equation}
\exp(U)(Y)=Y+[U, Y]+\frac{1}{2}[U,[U,Y]]\in\lieg_{-1},
\end{equation}
since $\textrm{ad}^3(U)(Y)=0$ by \eqref{compatibility_grading_eigenspaces} and $(c)$ of Proposition \ref{eigenvalues_on_grading_components}. 
The following result gives a description of the set $T(Z)$, which can be seen as a specialization to our setting of Kostant's description of $\algsl(2)$-triples in 
semisimple Lie algebras with the same nilpositive element (Theorem 3.6 of \cite{kostant.sl2.triples}).

\begin{proposition}\label{description_of_T(Z)}  
Suppose $\lieg=\lieg_{-1}\oplus\lieg_0\oplus\lieg_1$ is a complex or real $|1|$-graded semisimple Lie algebra and
let $Z\in\lieg_1$ be nonzero. Fix an element $X\in T(Z)$ and set $G_{0}^{[1]}=\exp(\lieg_0^{[1]})$.
Then $G_0^{[1]}$ acts simply transitively on $T(Z)$. In particular,
$T(Z)=\mbox{\emph{Ad}}(G_{0}^{[1]})(X) \subset\lieg_{-1}$.
\end{proposition}
\begin{proof}
Suppose $U \in \lieg_0^{[1]}$, and set $X'=\exp(U)X$, $A'=\exp(U)(A)$, and $Z'=\exp(U)Z$. 
Since $[U,Z]=0$ by $(e)$ 
of Proposition \ref{eigenvalues_on_grading_components}, we have $Z'=Z$. 
Because $\exp(U)$ acts by a Lie algebra homomorphism, $X'$ is in $T(Z)$, with $A'=[Z,X']$. 

Conversely, suppose $X' \in T(Z)$ and set $A'=[Z,X']$. Since $[A, Z]=2Z=[A', Z]$, 
we have $A-A'\in\mathfrak z(Z)\cap\lieg_0$. From $[Z, X-X']=A-A'$, we see that $A-A'\in[Z,\lieg_{-1}]$. Therefore, it follows from $(e)$ of
Proposition \ref{eigenvalues_on_grading_components} that $A-A' \in \lieg_0^{[1]}$. For $U=A-A'\in\lieg_0^{[1]}$ we therefore obtain
$$\exp(U)(A)=A+[A-A', A]=A+A'-A=A'.$$ Set $X''=\exp(U)(X)$. Then $X''$ and $X'$ are both elements of $T(Z)$ such that 
$[Z,X'']=A'=[Z,X']$. Hence, $X''-X'\in C(Z)$. But $X'-X''\in\lieg_{-1}$ is also in the $-2$-eigenspace of $A'$ 
on $\lieg_{-1}$, which implies by $(c)$ of Proposition \ref{eigenvalues_on_grading_components} that $X''=X'$.
Hence, $G_0^{[1]}$ acts transitively on $T(Z)$. 

It remains to show that the action is free.
Note that $\exp(U)(X)=X\in\lieg_{-1}^{[-2]}$ for some $U\in\lieg_{0}^{[1]}$ if and only if 
the elements $[U, X]\in\lieg_{-1}^{[-1]}$ and $[U,[U,X]]\in\lieg_{-1}^{[0]}$ are zero. 
Since $\textrm{ad}(X)$ induces an isomorphism  between 
$\lieg_0^{[1]}$ and $\lieg_{-1}^{[-1]}$ by $(a)$ of Proposition \ref{eigenvalues_on_grading_components}, 
these elements vanish if and only if $U=0$. Since the action is transitive, it now follows that the action is free.
\end{proof}

For simple $|1|$-graded Lie algebras, the description of $T(Z)$ in Proposition 
\ref{description_of_T(Z)} implies:

\begin{corollary} \label{intersection}
Suppose $\lieg=\lieg_{-1}\oplus\lieg_0\oplus\lieg_1$ is a complex or real $|1|$-graded simple Lie algebra, and
let $0 \neq Z \in \lieg_1$. Denote by $S$ the set of all elements in $\lieg_{-1}$ that are contained in the $-2$-eigenspace 
of a semisimple element $A=[Z,X]$ for some $X\in T(Z)$. Then the linear span of $S$ coincides with $\lieg_{-1}$. In particular, if $\alpha$ is 
an element of $\Lambda^2\lieg_{-1}^*$ such that $\alpha\lrcorner Y=0$ for all elements $Y\in S$, then $\alpha=0$.
\end{corollary}
\begin{proof} Fix $X\in T(Z)$ and set $A=[Z,X]$. By Proposition \ref{description_of_T(Z)},
$$S=\{\exp(\lieg_0^{[1]})(Y): Y\in\lieg_{-1}^{[-2]}\}.$$ 
Note that for any $U\in\lieg_{0}^{[1]}$ and $Y\in\lieg_{-1}^{[-2]}$,
\begin{equation}
\exp(U)(Y)=Y+[U, Y]+\frac{1}{2}[U[U,Y]]\in\lieg_{-1},
\end{equation}
where $[U,Y]\in\lieg_{-1}^{[-1]}$ and $\frac{1}{2}[U,[U,Y]]\in\lieg_{-1}^{[0]}=C(Z)$. By definition $S$ contains $\lieg_{-1}^{[-2]}$, which obviously
equals $\exp(0)(\lieg_{-1}^{[-2]})$. By $(a)$ of Proposition \ref{eigenvalues_on_grading_components} the map $\textrm{ad}(X)$ induces 
an isomorphism $\lieg_{0}^{[1]}\cong\lieg_{-1}^{[-1]}$. Since $\exp(U)(X)-\exp(-U)(X)=2[U,X]$ for any $U\in\lieg_{0}^{[1]}$, 
we therefore deduce that $\lieg_{-1}^{[-1]}$ is contained in the span of $S$. This immediately implies that also all elements in 
$\lieg_{-1}^{[0]}$ of the form 
\begin{equation}\label{e1}
[U,[U,Y]]\quad \textrm{ for } U\in\lieg_{0}^{[1]} \textrm{ and } Y\in\lieg_{-1}^{[-2]}
\end{equation}
are contained in the span of $S$. 

It remains to show that the elements in (\ref{e1}) span $\lieg_{-1}^{[0]}=C(Z)$. 
Let $U_1$ and $U_2$ be elements in $\lieg_{0}^{[1]}$ and $Y$ be an element in $\lieg_{-1}^{[-2]}$. Then
by the Jacobi identity and $(e)$ of Proposition \ref{eigenvalues_on_grading_components} we have
$[U_1,[U_2, Y]]=[U_2,[U_1,Y]]$, which implies that 
$$[U_1+U_2, [U_1+U_2, Y]=[U_1,[U_1, Y]]+2[U_1,[U_2,Y]]+[U_2,[U_2, Y]].$$
Therefore, $[U_1,[U_2,Y]]$ can be written as a linear combination of elements 
of the form (\ref{e1}).  Since $[\lieg_{0}^{[1]},\lieg_{-1}^{[-2]}]=\lieg_{-1}^{[-1]}$ by $(a)$ of 
Proposition \ref{eigenvalues_on_grading_components}, it is therefore sufficient to show that 
$[\lieg_{0}^{[1]},\lieg_{-1}^{[-1]}]=\lieg_{-1}^{[0]}$. By the Jacobi identity one verifies directly that
\begin{equation*}
[\lieg^{[-2]},\lieg^{[2]}]\oplus[\lieg^{[-1]},\lieg^{[1]}]\oplus\bigoplus_{j\neq 0}\lieg^{[j]}
\end{equation*}
is an nonzero ideal in $\lieg$. Since $\lieg$ is simple, we therefore deduce that $[\lieg^{[-2]},\lieg^{[2]}]$
and $[\lieg^{[-1]},\lieg^{[1]}]$ span $\lieg^{[0]}$. Hence, $(d)$ of Proposition \ref{eigenvalues_on_grading_components}
implies in particular that $[\lieg^{[-1]}_{-1},\lieg^{[1]}_{0}]=\lieg_{-1}^{[0]}$.
\end{proof}

\begin{remark} Suppose $\lieg$ is a simple Lie algebra equipped with a $|1|$-grading. Then $\lieg^0 = \liep$ corresponds to a single simple root, so $\mathfrak z(\lieg_0)$ is $1$-dimensional (see \cite[Thm 3.2.1, Prop 3.2.7]{cap.slovak.book.vol1}). Note that for any $0\neq Z\in\lieg_1$ and $X\in T(Z)$, 
the action of the corresponding semisimple element $A=[Z,X]\in\lieg_0$ on $\lieg$ might not have $1$ as an eigenvalue, in which case $-1$ is also not an eigenvalue.  If 
$\lieg^{[1]}=\lieg^{[-1]}=\{0\}$, then $(d)$ of Proposition \ref{eigenvalues_on_grading_components} gives $\lieg_0=\lieg_0^{[0]}$, and hence 
$A\in\mathfrak z(\lieg_0)$. Then $A$ must be twice the grading element, and $\lieg_{-1}=\lieg_{-1}^{-2}$.
\end{remark}

Recall that for any irreducible parabolic geometry the curvature has values in the $P$-module
$$\mathbb W=\Lambda^2(\lieg/\mathfrak p)^*\otimes\lieg\cong \Lambda^2\lieg_{1}\otimes\lieg,$$ 
Since $(\lieg/\mathfrak p)^*$ is completely reducible as a $P$-module and isomorphic to $\lieg_{-1}^*$ as $G_0$-module, 
we can make $\lieg_{-1}^*$ into a $P$-module isomorphic to $(\lieg/\liep)^*$ by defining the action 
of $P_+$ on $\lieg_{-1}^*$ to be trivial. In this way, $\Lambda^2\lieg_{-1}^*\otimes\lieg$  can be viewed as a $P$-module, 
which is isomorphic to $\mathbb W$. Moreover, identifying $\Lambda^2\lieg_{-1}^*\otimes\lieg$ with $\mathbb W$,
the grading on $\lieg$ induces a vector space grading on $\mathbb W$
by homogeneities as follows: 
\begin{equation}\label{grading_on_W}
\mathbb W=\mathbb W_1\oplus\mathbb W_2\oplus \mathbb W_2,
\end{equation}
where $\mathbb W_i=\Lambda^2\lieg_{-1}^*\otimes\lieg_{-2+i}$ for $i=1,2, 3$.
Note that $\mathbb W_i$ is a $G_0$-submodule of $\mathbb W$. 

We shall later need
some facts about the eigenspace decomposition of $\mathbb W$ 
under the action of $A=[Z,X]$, for $Z\in\lieg_1$ and $X\in T(Z)$. Since $A\in\lieg_0$, the action of 
$A$ preserves the decomposition \eqref{grading_on_W} of $\mathbb W$.  The eigenspace 
decomposition of each homogeneous component of $\mathbb W$ with respect to the action of $A$ 
is therefore of the form:
\begin{equation*}
\mathbb W_i=\bigoplus_{j\in\bold Z}\mathbb W^{[j]}_i\qquad 
\textrm{ with }\quad \mathbb W^{[j]}_i=\mathbb W_i\cap\mathbb W^{[j]} \qquad i = 1,2,3.
\end{equation*}

\begin{proposition}\label{eigenspaces_on_curvature} 
Suppose $\lieg=\lieg_{-1}\oplus\lieg_0\oplus\lieg_1$ is a complex or real $|1|$-graded semisimple Lie algebra, and
let $Z\in\lieg_1$ be nonzero. Fix an element $X\in T(Z)$ and set $A=[Z,X]$. Then for the action of $A$ on $\mathbb W$, the following holds:
\begin{itemize}
\item[(a)] $\mathbb W^{[st]}\cap(\Lambda^2\lieg_{-1}^*\otimes\bigoplus_{i\geq -1}\lieg^{[i]})\subset\ker(\llcorner Y)$ 
for all $Y\in\lieg_{-1}^{[-2]}$
\item[(b)] $\mathbb W^{[ss]}\subset\ker(\llcorner Y)$ and $Z\ker(\llcorner Y)\subset\ker(\llcorner Y)$
for all $Y\in\lieg_{-1}^{[-2]}$
\item[(c)] $\mathbb W^{[1]}_3\subset\ker(\llcorner Y)$
for all $Y\in\lieg_{-1}^{[-2]}$,
\end{itemize}
\end{proposition}
\begin{proof}
$(a)$ Suppose first that $\alpha$ is an element of $\mathbb W^{[-j]}\cap(\Lambda^2\lieg_{-1}^*\otimes\bigoplus_{i\geq -1}\lieg^{[i]})$ 
for some integer $j\geq 0$. Recall that by $(c)$ of Proposition \ref{eigenvalues_on_grading_components}  the possible 
eigenvalues of $A$ on $\lieg_{-1}$ are $-2$, $-1$, and $0$. Now let $Y \in \lieg_{-1}^{[-2]}$ and $V \in \lieg_{-1}^{[-\ell]}$ for some integer $0\leq\ell\leq 2$. 
Then 
$$-j\alpha(Y,V)=(A\cdot\alpha)(Y,V)=[A,\alpha(Y,V)]-\alpha(AY, V)-\alpha(Y, AV)$$
and so
\begin{equation}\label{f1}
[A,\alpha(Y, V)]=-(j+2+\ell)\alpha(Y,V).
\end{equation}
But $-(j+2+\ell)\leq -2$, which contradicts the assumption that $\alpha$ has values in $\bigoplus_{i\geq -1}\lieg^{[i]}$, 
unless $\alpha(Y,V)= 0$. Now $(a)$ follows by linearity.
\\$(b)$ 
Note that for $\alpha\in\mathbb W^{[-j]}\subset \mathbb W^{[ss]}$ with $j>0$ some integer, the equation \eqref{f1} takes the form
$[A,\alpha(Y, V)]=-(j+2+\ell)\alpha(Y,V)$ with $-(j+2+\ell)< -2$. Since all eigenvalues on $\lieg$ are $\geq-2$ by Proposition 
\ref{eigenvalues}, we conclude that $\alpha(Y,V)=0$. Hence, by linearity the first statement of $(b)$ holds.
Since $Z$ acts trivially on $\Lambda^2\lieg_{-1}^*$, its action on $\mathbb W$ preserves the space of forms in $\mathbb W$ that vanish 
upon insertion of an element of $\lieg_{-1}^{[-2]}$; this completes $(b)$. 

For $(c)$, note that for $\alpha\in\Bbb W_3^{[1]}$, the equation (\ref{f1}) reads as 
$$[A,\alpha(Y, V)]=-(1+\ell)\alpha(Y,V) \qquad l \geq 0,$$
which implies that $\alpha(Y,V)=0$, since $A$ has nonnegative eigenvalues on $\lieg_1$ by $(d)$ of Proposition 
\ref{eigenvalues_on_grading_components}. Again by linearity, $\alpha\in\ker(\llcorner Y)$
for all $Y\in\lieg_{-1}^{[-2]}$.
\end{proof}

\subsection{Fundamental derivative}
In the sequel we will use the notion of the fundamental derivative of parabolic geometries.
Let us therefore recall its definition and basic properties; for more details see for instance
\cite{cap.slovak.book.vol1}, \cite{calderbank.diemer} or \cite{me.frobenius}. Suppose 
$(B\stackrel{\pi}{\rightarrow} M,\omega)$ is a parabolic geometry. Consider an associated vector 
bundle $V=B\times_P\mathbb V$ and identify its space of sections with
$P$-equivariant smooth functions $C^\infty(B,\mathbb V)^P$. Differentiating equivariant
functions $f\in C^\infty(B,\mathbb V)^P$ in the direction of $\omega$-constant vector fields gives 
rise to a differential operator
\begin{center}
$D: C^\infty(B,\mathbb V)^P\rightarrow C^\infty(B,\lieg^*\otimes \mathbb V)^P$
\\
$Df(b)(X)=(\tilde X\cdot f)(b)$,
\end{center}
where $\tilde X$ denotes the $\omega$-constant vector field with value $X\in\lieg$. 
It is called the \emph{fundamental derivative} on $V$.

Since the fundamental derivative is defined on any associated bundle, we can iterate it and define the $m^{\mathrm{th}}$ fundamental derivative on $V$ by
\begin{center}
$D^m: C^\infty(B,\mathbb V)^P\rightarrow C^\infty(B,\otimes^m\lieg^*\otimes \mathbb V)^P$
\\
$D^mf(b)(X_1,\ldots,X_m)=(\tilde X_1 \cdots \tilde X_m\cdot f)(b)$.
\end{center}
Observe first that for any $X\in\liep$, the corresponding $\omega$-constant vector field 
$\tilde X$ is the fundamental vector field generated by the right action of $X$ on $B$. Thus 
\begin{equation}\label{D_Property 1}
Df(b)(X)=(\tilde{X} \cdot f)(b)=-X(f(b))\quad \textrm{ for all } X\in\liep,
\end{equation}
which implies that for all $m\geq 1$, $X_1,\ldots, X_{m-1}\in\lieg$, and $X_m\in\liep$,
\begin{equation}\label{D_Property 2}
D^mf(b)(X_1,\ldots,X_m)=-X_m(D^{m-1}(b)(X_1,\ldots,X_{m-1}))
\end{equation}

Since the curvature $\kappa$ of a Cartan geometry satisfies 
$\omega^{-1}([X_1, X_2]-\kappa(X_1, X_2))=[\tilde X_1,\tilde X_2]$, the second fundamental derivative satisfies the following property for all $X_1, X_2\in\lieg$, which is sometimes referred 
to as the \emph{Ricci identity of the fundamental derivative}:
\begin{equation}\label{Ricci_identity}
D^2f(X_1,X_2)-D^2f(X_2, X_1)=Df([X_1,X_2])-Df(\kappa(X_1,X_2)).
\end{equation}
Because the curvature is semibasic, it follows that 
\begin{equation}\label{Ricci_identity2}
D^2f(X_1,X_2)-D^2f(X_2, X_1)=Df([X_1,X_2]) \qquad \forall X_1\in\lieg, X_2\in\liep.
\end{equation}

Recall that a vector field $\tilde\eta\in\mathfrak{inf}(B,\omega)$ commutes 
with all $\omega$-constant vector fields. Hence, if $\tilde\eta\cdot f=0$ for some equivariant function 
$f\in C^\infty(B,\mathbb V)^P$, then $\tilde\eta\cdot (D^mf)=0$ for all 
$m\geq 0$, which in turn is equivalent to
$$ (D^m f)(\varphi^t_{\tilde{\eta}} b)  = (D^mf)(b)\quad\textrm{ for all } b\in B.$$

Since any infinitesimal automorphism of a Cartan geometry 
preserves the curvature function, 
\begin{equation}\label{Inf_aut_preserves_curvature}
\tilde\eta\cdot D^m\kappa=0\quad\textrm{ for all } m\geq0, \ \tilde\eta\in\mathfrak{inf}(B,\omega)
\end{equation}
Assume now that $\tilde\eta\in\mathfrak{inf}(B,\omega)$ has a higher-order zero at 
$x_0\in M$, and let $Z$ be the isotropy with respect to $b_0 \in \pi^{-1}(x_0)$.
Then equations (\ref{D_Property 1}) and (\ref{Inf_aut_preserves_curvature}) imply that at the point $b_0$,
\begin{equation}\label{Inf_aut_preserves_curvature2}
0=(\tilde\eta\cdot D^m\kappa)(b_0)=-Z(D^m\kappa(b_0)).
\end{equation}
Note that for any $X\in\lieg$, the derivative $D^m\kappa(b_0)(X,\ldots,X)$ equals the $m^{\mathrm{th}}$ 
derivative at $0$ of $\kappa$ along the exponential curve $\tilde{\gamma}_X = \exp_{b_0}(sX)$---that is,
$$D^m\kappa(b_0)(X,\ldots,X)=\left. \frac{d^m}{ds^m}\right|_{s=0}\kappa(\tilde{\gamma}_X(s).$$
 
The properties of the fundamental derivative mentioned above and the identity (\ref{Inf_aut_preserves_curvature2}) 
give the following information on the derivatives at $0$ of $\kappa$ along the curves in $\mathcal T(\alpha)$:

\begin{proposition}\label{derivatives_of_curvature_at_b0} 
Suppose $(B\stackrel{\pi}{\rightarrow} M,\omega)$ is a parabolic geometry of type $(\lieg, P)$. 
Assume $0 \neq \tilde\eta\in\mathfrak{inf}(B,\omega)$ has a higher-order zero at $x_0\in M$. 
Let $Z$ be the istropy with respect to $b_0 \in \pi^{-1}(x_0)$. Assume in addition 
that there exists $X\in T(Z)\subset\lieg_-$ such that $A=[Z,X]\in\lieg_0$. Then:

\begin{enumerate}
\item[(a)] $Z\kappa(b_0)=0$
\item[(b)] $Z(D\kappa(b_0)(X))=-A\kappa(b_0)$
\item[(c)] $Z(D^2\kappa(b_0)(X,X))=-2(A+Id)(D\kappa(b_0)(X))$
\end{enumerate}
\end{proposition}
\begin{proof}
The identity in $(a)$ is just the identity (\ref{Inf_aut_preserves_curvature2}) for $m=0$. 
By (\ref{Inf_aut_preserves_curvature2}) 
and equation (\ref{D_Property 1}),
$$0=(ZD\kappa(b_0))(X)=Z(D\kappa(b_0)(X))-D\kappa(b_0)([Z,X])=Z(D\kappa(b_0)(X))+A\kappa(b_0),$$ 
which proves $(b)$. For $(c)$ note that (\ref{Inf_aut_preserves_curvature2}) implies that
$$0=(ZD^2\kappa(b_0))(X,X)=Z(D^2\kappa(b_0)(X,X))-D^2\kappa(b_0)([Z,X],X)-D^2\kappa(b_0)(X,[Z,X]).$$
By the Ricci identity (\ref{Ricci_identity2}) and equation (\ref{D_Property 2}), we obtain
\begin{align*}
&D^2\kappa(b_0)([Z,X],X)+D^2\kappa(b_0)(X,[Z,X])=\\
&2D^2\kappa(b_0)(X,A)+D\kappa(b_0)([A,X])=-2A(D\kappa(b_0)(X))-2D\kappa(b_0)(X).
\end{align*}
and hence $(c)$ holds. 
\end{proof}

\subsection{Curvature vanishing result for irreducible parabolic geometries}
Suppose now that $(B\stackrel{\pi}{\rightarrow} M,\omega)$ is a normal irreducible parabolic 
geometry. Assume $0 \neq \eta\in\mathfrak{inf}(M)$ has a higher-order zero at $x_0\in M$, and that the isotropy with respect to $b_0 \in \pi^{-1}(x_0)$ is $Z$. 


Recall from \eqref{ker_partial*} that the harmonic curvature can be viewed as a $G_0$-equivariant function
$$\hat\kappa: B\rightarrow \widehat{\mathbb W}\subset\ker(\partial^*)\subset \mathbb W,$$
which is constant along the fibers of $B\rightarrow B/P_+$. 


In accordance with the grading of $\mathbb W$ into homogeneous components as in \eqref{grading_on_W},
we write $\kappa_i$, $i=1,2,3$, for the components of the curvature.
Since the grading \eqref{grading_on_W} on $\mathbb W$ is $G_0$-invariant,
it induces a corresponding grading on the $G_0$-submodule $\widehat{\mathbb W}$:
\begin{equation}\label{harm_curv_components}
\widehat{\mathbb W}=\widehat{\mathbb W}_1 \oplus\widehat{\mathbb W}_2\oplus\widehat{\mathbb W}_3 \ \ \mbox{where} \ \  \widehat{\mathbb W}_i=\widehat{\mathbb W}\cap\mathbb W_i, \ i=1,2,3.
\end{equation}
Decompose the harmonic curvature accordingly into homogeneous components $\hat\kappa_i$, for $i=1,2,3$.
Note that $A = [Z,X]$, for any $X \in T(Z)$, lies in $\lieg_0$, so the action of $A$ preserves each component
$\widehat{\mathbb W}_i$. Hence, the decomposition of 
$\widehat{\mathbb W}_i$ into eigenspaces of $A$ is of the form:
\begin{equation}\label{hat_W_i^j}
\widehat{\mathbb W}_i=\bigoplus_{j\in\bold Z}\widehat{\mathbb W}^{[j]}_i\qquad 
\textrm{ with }\quad \widehat{\mathbb W}^{[j]}_i=\widehat{\mathbb W}\cap\mathbb W^{[j]}_i, \ i=1,2,3 
\end{equation}

\begin{proposition}\label{harmonic_curvature_along_distinguished_curves1} 
Suppose $(B\stackrel{\pi}{\rightarrow} M,\omega)$ is a normal irreducible parabolic geometry of type $(\lieg, P)$ with $\lieg$ simple.  Let
$0 \neq \eta\in\mathfrak{inf}(M)$ have a higher-order zero at $x_0\in M$, and let the isotropy with respect to $b_0\in \pi^{-1}(x_0)$ be $Z$. Then for any $X\in T(Z)$, 
$$(\hat{\kappa} \circ \tilde{\gamma}_X)(s) \in \widehat{\mathbb W}^{[ss]}\subset \Bbb{W}^{[ss]} \qquad \textrm{ for all } s\in I,$$ 
where $I$ is the domain of $\tilde{\gamma}_X$. Therefore, for all $s\in I$,
\begin{itemize}
\item $\hat{\kappa}_1(\tilde{\gamma}_X(s)) = s w_{1}^{[-1]} + s^2 w_{1}^{[-2]}$ 
\item $\hat\kappa_{2} (\tilde{\gamma}_X(s)) = s w_2^{[-1]}$  
\item $\hat\kappa_{3} (\tilde{\gamma}_X(s)) =0$,
\end{itemize}
for some elements $w_{1}^{[-1]}\in\widehat{\mathbb W}^{[-1]}_1$, $w_1^{[-2]}\in\widehat{\mathbb W}^{[-2]}_1$ 
and $w_2^{[-1]}\in\widehat{\mathbb W}^{[-1]}_2$.
\end{proposition}

\begin{proof}
By Proposition \ref{derivatives_of_curvature_at_b0} $(a)$, $Z\kappa(b_0)=0$. Since $Z$ acts trivially 
on $(\lieg/\liep)^*\cong\lieg_{-1}^*$, this implies that $\kappa(b_0)$ has values in $\mathfrak z(Z)$. 
From \eqref{center_decomposition}, $\kappa(b_0)$ is an element of $\Lambda^2\lieg_{-1}^*\otimes\bigoplus_{i\geq 0}\lieg^{[i]}$.
By $A$-invariance of $\widehat{\Bbb{W}}$, also
$$\hat\kappa(b_0)\in\Lambda^2\lieg_{-1}^*\otimes\bigoplus_{i\geq 0}\lieg^{[i]}.$$ 
By \eqref{tau_along_special_curves}, $\hat{\kappa}(\tilde{\gamma}_X(s))\in\widehat{\mathbb W}^{[st]}\subset \Bbb{W}^{[st]}$ for all $s\in I$. 
Proposition \ref{eigenspaces_on_curvature} $(a)$ give that, for all elements $Y\in\lieg_{-1}^{[-2]}$,
$$\hat\kappa(b_0)\lrcorner Y=0.$$ 
Letting $X$ vary over $T(Z)$ and applying Corollary \ref{intersection} gives that $\hat\kappa(b_0)=0$. Hence, \eqref{tau_along_special_curves} implies
$$(\hat{\kappa} \circ \tilde{\gamma}_X)(s) \in \widehat{\mathbb W}^{[ss]}\subset \Bbb{W}^{[ss]}$$ for all $s\in I$ as claimed. 
By $(c)-(e)$ of Proposition \ref{eigenvalues_on_grading_components} the possible negative 
eigenvalues on $\mathbb W_1$ are $-2$ and $-1$, on $\mathbb W_2$ just $-1$, and on $\mathbb W_3$ all eigenvalues 
are nonnegative. Hence, the second statement follows immediately from \eqref{hat_W_i^j} and Proposition 
\ref{prop.s.to.the.k}.
\end{proof}

It is a consequence of the Bianchi identities of the Cartan curvature $\kappa$ that the lowest homogeneous component of $\kappa$ 
always coincides with the lowest homogeneous component of the harmonic curvature (see \cite[Theorem 3.1.12]{cap.slovak.book.vol1}). 
This lowest homogeneity component of the curvature 
is $\kappa_1$, which coincides with the torsion of the geometry. Thus 
$\kappa_1=\hat\kappa_1$.
Now Proposition \ref{harmonic_curvature_along_distinguished_curves1} implies that
\begin{equation*}
\kappa_{1} (\tilde{\gamma}_X(s)) =\hat\kappa_{1} (\tilde{\gamma}_X(s))=
s w_{1}^{[-1]} + s^2 w_{1}^{[-2]}. 
\end{equation*}
Let us remark that $\kappa_1$ might be identically zero, implying that $\hat\kappa_1$ is identically zero. In this case, 
the Bianchi identity again gives $\kappa_2=\hat\kappa_2$ \cite[Theorem 3.1.12]{cap.slovak.book.vol1}. If  $\kappa_2$
is also identically zero, then similarly $\kappa_3=\hat\kappa_3$.

We now complete the proof of Theorem \ref{thm.harmonic.vanishing}.

\begin{proof}[Proof of Theorem \ref{thm.harmonic.vanishing}]
Fix a point $b_0\in\pi^{-1}(x_0)$ to identify the isotropy of $\eta\in\mathfrak{inf}(M)$ with $Z\in\liep_+$. Fix $X\in T(Z)$, and write $I$ for the domain of $\tilde\gamma_X$. From Proposition \ref{harmonic_curvature_along_distinguished_curves1}, $\kappa_1(b_0)=\hat\kappa_1(b_0)=0$, $\hat\kappa_2(b_0)=0$, and $\hat\kappa_3\circ\tilde\gamma_X\equiv 0$. The remaining claims of the theorem follow in 3 steps.
\\\textbf{1. Step}: \emph{$\kappa_{2} (b_0)=0$ and $w_{1}^{[-1]} = 0$ in the expression of 
$\kappa_1(\tilde{\gamma}_X(s))=\hat\kappa_1(\tilde{\gamma}_X(s))$ in Proposition 
\ref{harmonic_curvature_along_distinguished_curves1}.}
\\
From equation (\ref{eqn.sl2.holo}), $a_s(t)$ is a holonomy path for the flow $\{ \varphi_{\tilde\eta}^t \}$ at $\tilde{\gamma}_X(s) e^{- \frac{1}{s} Z}$, for $s > 0$ (and for $\{ \varphi^{-t}_{\tilde{\eta}} \}$ if $s < 0$). Therefore, 
$\kappa(\tilde{\gamma}_X(s) e^{-\frac{1}{s} Z})=e^{\frac{1}{s}Z}\kappa(\tilde{\gamma}_X(s))$ is in $\Bbb{W}^{[st]}$ for all $s \neq 0$. The action of 
$Z$ on $\mathbb W$ raises homogeneities by $1$, so, as an operator on $\mathbb W$, we have $e^{\frac{1}{s}Z} = 1 + \frac{1}{s} Z + \frac{1}{2 s^2} Z^2$. 
From the polynomial expression of $\kappa_1(\tilde{\gamma}_X(s))=\hat\kappa_1(\tilde{\gamma}_X(s))$ in Proposition 
\ref{harmonic_curvature_along_distinguished_curves1}, the component of homogeneity $2$ of $\kappa$ at $\tilde{\gamma}_X(s) e^{- \frac{1}{s} Z}$ is 
\begin{equation}\label{eq2}
 \kappa_{2} (\tilde{\gamma}_X(s)) + \frac{1}{s} Z\kappa_{1}(\tilde{\gamma}_X(s)) =  \kappa_{2} (\tilde{\gamma}_X(s)) + Z w_{1}^{[-1]} + s Z w_{1}^{[-2]},
 \end{equation}
 where $Z w_{1}^{[-1]}\in \Bbb{W}^{[1]}_2$ and $Z w_{1}^{[-2]}\in\Bbb{W}^{[0]}_2$.
Since the expression (\ref{eq2}) lies in $\Bbb{W}^{[st]}_2$ for all $s \neq 0$ and the term $Z w_{1}^{[-1]}$ on the right hand side
is unstable, $Z w_{1}^{[-1]}$ must cancel with a term of $\kappa_{2} (\tilde{\gamma}_X(s))$.  Therefore, by continuity in $s$, 
\begin{equation}\label{eq3}
\kappa_{2} (\tilde{\gamma}_X(s))  \in - Z w_{1}^{[-1]} + \Bbb{W}^{[st]}_{2} \qquad \forall s \in I
\end{equation}
We then deduce from Proposition \ref{eigenvalues_on_grading_components} $(e)$ and Proposition \ref{eigenspaces_on_curvature} $(a)$ and $(b)$
that the expression (\ref{eq3}) lies in $\ker(\llcorner Y)$ for all $Y\in\lieg_{-1}^{[-2]}$. Varying $X\in T(Z)$ and applying Corollary \ref{intersection} gives $\kappa_{2} (b_0)=0$ as desired.
Moreover, it follows that
$$Z w_{1}^{[-1]}= 0\quad \textrm{ and }\quad \kappa_{2}(\tilde{\gamma}_X(s)) \in \Bbb{W}_{2}^{[st]}\quad \textrm{ for all } s\in I.$$  
From Proposition \ref{eigenvalues_on_grading_components}, all eigenvalues of $A$ on $\lieg_{-1}^*$ are nonnegative, 
so $w_{1}^{[-1]}\in\widehat{\Bbb{W}}_{1}^{[-1]}\subset\Lambda^2\lieg_{-1}^*\otimes\lieg_{-1}$ implies that $w_{1}^{[-1]}$ takes values in $\lieg_{-1}^{[ss]}$.  
On the other hand, $Z$ acts trivially on $\lieg_{-1}^*$, so $Z w_{1}^{[-1]} = 0$ implies that $w_{1}^{[-1]}$ takes values in $C(Z)$, which by $(c)$ 
of Proposition \ref{eigenvalues_on_grading_components} equals $\lieg_{-1}^{[0]}$. 
Therefore $w_{1}^{[-1]}= 0$, and  
\begin{equation}\label{eq5}
\kappa_{1} (\tilde{\gamma}_X(s)) = s^2 w_{1}^{[-2]}\quad \textrm{ for all } s\in I.
\end{equation}

\textbf{2. Step}: \emph{$\kappa_{3} (b_0)=0$, and $w_{1}^{[-2]} = 0$ in the expression of 
$\kappa_1(\tilde{\gamma}_X(s))=\hat\kappa_1(\tilde{\gamma}_X(s))$ in Proposition 
\ref{harmonic_curvature_along_distinguished_curves1}.  Thus $\kappa_{1} \circ\tilde{\gamma}_X\equiv0$.}
\\From \eqref{eq5}, $D^2\kappa_1(b_0)(X,X)=2 w_{1}^{[-2]}$.
By Proposition \ref{derivatives_of_curvature_at_b0} $(c)$, 
$$-2(A+Id)(D\kappa_2(b_0)(X))=Z(D^2\kappa_1(b_0)(X,X))=2Zw_{1}^{[-2]}\in\Bbb W^{[0]}_2,$$ which implies that
\begin{equation}\label{eq6}
D\kappa_2(b_0)(X)\in-Zw_{1}^{[-2]}+\Bbb W^{[-1]}_2.
\end{equation}
By Proposition \ref{derivatives_of_curvature_at_b0} $(b)$, 
\begin{equation}\label{eq4}
Z(D\kappa_{2}(b_0) (X)) = - A \kappa_{3}(b_0). 
\end{equation}
Since $Z^2w_{1}^{[-2]}$ belongs to $\Bbb{W}^{[2]}_3$
and $Z\Bbb W^{[-1]}_2\subset \Bbb W^{[1]}_3$, we conclude from (\ref{eq6}) and (\ref{eq4})
that 
$$\kappa_3(b_0)\in- \frac{1}{2} Z^2w_{1}^{[-2]} +  \Bbb W^{[1]}_3 +\Bbb{W}^{[0]}_{3}.$$
By Proposition \ref{eigenvalues_on_grading_components} $(d)$, all eigenvalues of $A$ on $\lieg_1$ are nonnegative.  Then $(a)-(c)$ of Proposition \ref{eigenspaces_on_curvature} give
$$\kappa_{3}(b_0)\in \ker(\llcorner Y)\qquad \textrm{ for all } Y\in\lieg_{-1}^{[-2]}.$$
Varying $X\in T(Z)$ and applying Corollary \ref{intersection} gives $\kappa_3(b_0)=0$, as claimed. In particular,  
$Z^2w_{1}^{[-2]}= 0$.  Since all eigenvalues of $A$ on $\lieg_{-1}^*$ are nonnegative, 
the element $w_{1}^{[-2]}$ must have values in $\lieg_{-1}^{[-2]}$.
Because $Z$ acts trivially on $\lieg_{-1}^*$, the identity $Z^2w_{1}^{[-2]}= 0$ implies that the image of 
$w_{1}^{[-2]}$ in must be annihilated by $Z^2$. Hence, we conclude from
$(b)$ of Proposition \ref{eigenvalues_on_grading_components} that $w_{1}^{[-2]}=0$.

\textbf{3. Step} \emph{$\hat\kappa_2\circ\tilde\gamma_X\equiv 0$.}
\\Since $w_{1}^{[-2]}=0$, we see from (\ref{eq6}) that $D\kappa_2(b_0)(X)=v_{2}^{[-1]}$ equals some $v_{2}^{[-1]}\in\Bbb W^{[-1]}_2$. 
Recall that by Proposition \ref{harmonic_curvature_along_distinguished_curves1}, $\hat\kappa_2(\tilde\gamma_X(s))=sw_{2}^{[-1]}\in \widehat{\mathbb W}^{[-1]}_2\subset \Bbb W^{[-1]}_2$. Since the decomposition \eqref{ker_partial*} of $\ker(\partial^*)$ is invariant under the action of $A\in\lieg_0$, we have $v_{2}^{[-1]}=w_{2}^{[-1]}+\bar{w}_{2}^{[-1]}$, 
where $\bar{w}_{2}^{[-1]}\in\textrm{im}(\partial^*)$. 
By $(c)$ and $(e)$ of Proposition \ref{eigenvalues_on_grading_components}, the elements $w_{2}^{[-1]}$ and $\bar{w}_{2}^{[-1]}$ must have values 
in $\lieg_0^{[-1]}$. On the other hand, since $\kappa_3(b_0)=0$, the equation (\ref{eq4}) implies that $Zv_{2}^{[-1]}=0$ and hence, as above, that $Z$ annihilates the image of $v_2^{[-1]}$. Then by 
Proposition \ref{eigenvalues_on_grading_components} $(a)$, $v_{2}^{[-1]}=0$.  
Since $w_{2}^{[-1]}$ and $\bar{w}_{2}^{[-1]}$ are linearly independent, we must have in particular that $w_{2}^{[-1]}=0$, which completes the proof.
\end{proof}

\section{Rigidity results for irreducible parabolic geometries}
\label{sec.applications}
In this section we apply Theorem \ref{thm.harmonic.vanishing} to prove new rigidity results for various irreducible parabolic geometries admitting strongly essential flows.  We prove general curvature vanishing results for flows with smoothly isolated zeroes in Proposition \ref{Case_isolated_fixed_point}, and for flows with maximal strongly fixed sets in Proposition \ref{C(Z)_max_prop}.  These are applied to various specific geometries in Corollaries \ref{Cor_isolated_zero} and \ref{Cor_C(Z)_max}, respectively.  We also prove Theorem \ref{thm.almost.cproj} and the quaternionic counterpart, Theorem \ref{thm.almost.quat}.

\subsection{Classification of irreducible parabolic geometries}\label{classification}

The complete classification of parabolic subalgebras with abelian nilradical in simple Lie algebras can be found in \cite{cap.slovak.book.vol1}. In Sections \ref{typeA}--\ref{typeD} below, we describe the irreducible parabolic geometries modeled on $(\lieg, P)$ with $\lieg$ simple, neglecting only a few obscure examples.  We shall write $P$ for any parabolic subgroup with Lie algebra $\liep$, but we will classify geometric types in $\liep_+$ assuming that $P$ is maximal with Lie algebra $\liep$.  For isogenous $P' <P$, geometries modeled on $(\lieg,P')$ have some additional geometric data on the underlying geometric structure, such as an orientation.  Since our results concern the local geometry around a higher-order zero, finer choices of $P$ will not make a difference.
 
\subsubsection{Type $A_n$}\label{typeA}
Assume $n\geq2$. In $\lieg=\algsl(n+1,\bold{K})$ for $\bold{K}=\BR, \BC$ or $\bold{H}$ there exist up to conjugation $n$ parabolic subalgebras with abelian nilradical, which can be realized as the stabilizers $\liep(p)$ in $\lieg$ of $\bold{K}^p\subset \bold{K}^{n+1}$ for $1\leq p\leq n$. 
Fix $p\geq 1$ and set  $\liep=\liep(p)$ and $q=n+1-p$. Identifying the Levi factor $\lieg_0 \subset \liep$ with the subalgebra preserving the decomposition  $\bold{K}^p\oplus\bold{K}^q=\bold{K}^{n+1}$, 
$\lieg$ can be identified with a $|1|$-graded Lie algebra of the following form
\begin{equation}\label{Grassmannian_grading}
\lieg=\lieg_{-1}\oplus\lieg_0\oplus\lieg_1=\textrm{L}_{\bold{K}}(\bold{K}^p,\bold{K}^q)\oplus\mathfrak s(\alggl(p,\bold{K})\times\alggl(q,\bold{K}))
\oplus \textrm{L}_{\bold{K}}(\bold{K}^q,\bold{K}^p),
\end{equation}
where $\liep=\lieg_0\oplus\lieg_1$. 
The action of $(A,B)\in\lieg_0$ on $X\in\lieg_{-1}$ is given by $(A,B)X=BX-XA$; on $Z\in\lieg_1$, it is $(A,B)Z=AZ-ZB$.  The bracket $\lieg_1\times\lieg_{-1}\rightarrow\lieg_0$ is $[Z,X]=(ZX,-XZ)$.

We will assume below that $p\leq q$, since $q\leq p$ leads to isomorphic geometries. 
The following table lists the parabolic geometries corresponding to the pair $(\lieg,P(p))$ and the possible types of isotropies of non-trivial infinitesimal automorphisms at higher order zeroes.

\begin{table}[h]
\begin{tabular}{|p{18mm}|c|c|}
\hline
 & Geometric structure  & Number of types of isotropies\\
 &                                         & of strongly essential flows \\
\hline\hline
$p=1$ & projective, almost c-projective, & $1$\\
          & and almost quaternionic structures                   &         \\
\hline
$2\leq p\leq q$ & almost Grassmannian structures of type $(p,q)$ &  $p$, corresponding to $\bold{K}$-ranks \\
                            & and their complex and quaternionic analogues & of non-zero  maps in \\
                            &                                                                                        &$\lieg_1=\textrm{L}_{\bold{K}}(\bold{K}^q,\bold{K}^p)$ \\
\hline
\end{tabular}
\end{table}

Structures in the first row are infinitesimally modeled on ${\bf KP}^n$, and those in the second row on Grassmann varieties $\mbox{Gr}_{\bf K}(p,n+1)$.  For detailed descriptions of these geometries, with the exception of almost c-projective structures, we refer to Section 4.1 of \cite{cap.slovak.book.vol1}.  An \emph{almost c-projective structure} consists of an almost complex structure and a complex projective class of \emph{minimal} complex connections---for which the torsion is a multiple of the Nijenhuis tensor. If the complex structure is integrable, these are also called \emph{h-projective structures}; see \cite{matveev.rosemann}, \cite{Hrdina}, and \cite{cald.east.mat.neu}.

\subsubsection{Type $B_n$} \label{typeB}
Let $m\geq 1$ be an integer. Up to conjugation, there is only one parabolic subalgebra with abelian nilradical in $\mathfrak o(2m+3,\BC)$.  It is the stabilizer of the highest weight line in the standard representation $\BC^{2m+3}$. This parabolic subalgebra admits real forms in all the
real forms $\mathfrak o(p+1,q+1)$ of $\mathfrak o(2m+3,\BC)$, where $p+q=2m+1$, which we shall all denote by $\liep$. Choosing a realization of the Levi factor of such $\liep$ as a subalgebra of $\liep$ gives the following $|1|$-grading on $\lieg$:
\begin{equation}\label{conformal_grading}
\lieg=\lieg_{-1}\oplus\lieg_0\oplus\lieg_1=\BR^{p,q}\oplus \mathfrak c\mathfrak o(p,q)\oplus(\BR^{p,q})^*,
\end{equation}
where $\liep=\lieg_0\oplus\lieg_1$ and $\mathfrak c\mathfrak o(p,q)=\BR\oplus \mathfrak o(p,q)$ denotes the conformal Lie algebra of signature $(p,q)$.
An element $(a,A)\in\lieg_0$ acts on $X\in\lieg_{-1}$ by $(a,A)X=-aX+AX$.  The Lie bracket between $Z\in\lieg_1$ and $X\in\lieg_{-1}$ is  
 \begin{equation}\label{conf_bracket_1}
 [Z,X]=(ZX, -XZ+ \mathbb I(XZ)^t\mathbb I),\quad \textrm{where}\,\, \mathbb I=\textrm{Id}_p\oplus -\textrm{Id}_q.
 \end{equation}
 
 Parabolic geometries of type $(\mathfrak o(p+1,q+1), P)$ correspond to \emph{semi-Riemannian conformal structures of signature $(p,q)$ on manifolds of dimension $2m+1$}. The standard homogeneous model 
 is the space of null lines in $\BR^{2m+3}$ equipped with its standard conformal structure, which is called the M\"obius space of signature $(p,q)$.
 
 For $pq\neq 0$ there are three possible types of isotropies of higher-order zeroes of conformal Killing fields.  These correspond to timelike, spacelike, and null elements in $\lieg_1=(\BR^{p,q})^*$.  Lightlike isotropies are called \emph{isotropic}, while spacelike and timelike isotropies are referred to as \emph{non-isotropic}. If $pq=0$, then there is only one type of isotropy.
 
\subsubsection{Type $C_n$}\label{typeC}
Let $n\geq 2$ be an integer and assume $\bold{K}=\BR$ or $\BC$. Let $\bold{K}^{2n}$ be equipped with the standard ${\bf K}$-linear symplectic form. The symplectic Lie algebra $\lieg=\mathfrak s\liep(2n,\bold{K})$ has up to conjugation a unique parabolic 
subalgebra $\liep$ with abelian nilradical. It is the stabilizer in $\lieg$ of the 
isotropic subspace $\mbox{span}\{ {\bf e}_1, \ldots, {\bf e}_n \} \subset \bold{K}^{2n}$. Identifying the Levi factor $\lieg_0 \subset \liep$ with the stabilizer of the complementary isotropic subspace $\mbox{span} \{ {\bf e}_{n+1} , \ldots, {\bf e}_{2n} \} \subset \bold{K}^{2n}$ gives an identification of $\lieg$ with the following $|1|$-graded Lie algebra:
\begin{equation}\label{grading_Lagrange_str}
\lieg=\lieg_{-1}\oplus\lieg_0\oplus\lieg_1=S^2\bold{K}^{n}\oplus\alggl(n,\bold{K})\oplus S^2\bold{K}^{n*},
\end{equation}
where $\liep=\lieg_0\oplus\lieg_1$. The action of an element $A\in\lieg_0$ on $X\in\lieg_{-1}$ corresponds to the action
of $\alggl(n,\bold{K})$ on $S^2\bold{K}^{n}$. The Lie bracket between $Z\in\lieg_1$ and $X\in\lieg_{-1}$ is $[Z,X]=-ZX$, where the right-hand side is the product of symmetric matrices.

Parabolic geometries of type $(\mathfrak s\liep(2n,\bold{K}), P)$ correspond to \emph{almost Lagrangean structures} (and to their complex analogues when ${\bf K} = {\bf C}$).  The standard homogeneous model is the variety of Lagrangean subspaces of $\bold{K}^{2n}$.
For ${\bf K} = {\bf R}$, such a structure on $M^{n(n+1)/2}$ is given by an isomorphism $TM \cong S^2 E$, for an auxiliary rank-$n$ vector bundle $E$. For more details about almost Langrangean structures see Section 4.1.11 of 
\cite{cap.slovak.book.vol1}.

For $\bold{K}=\bold{R}$, the possible types of isotropies of higher-order fixed points of non-trivial strongly essential flows  
are parametrized by the possible signatures $(p,q,r)$ of non-zero quadratic forms on $\BR^n$ of rank $p+q \neq 0$. If $\bold{K}=\bold{C}$, 
the possible types of isotropies correspond to the possible ranks of nonzero complex-bilinear quadratic forms.

\subsubsection{Type $D_n$} \label{typeD}
Let $m\geq 2$ be an integer. Up to conjugation, there are three 
parabolic subalgebras with abelian nilradical in $\mathfrak o(2(m+1),\BC)$. The first can be realized as the stabilizer $\liep$ of the highest weight line in the standard representation of $\mathfrak o(2(m+1),\BC)$; the others are the stabilizers $\lieq_1$ and $\lieq_2$ of the highest weight lines in the half-spin representations. Since $\lieq_1$ and $\lieq_2$ are apparently isomorphic Lie algebras, we restrict to $\lieq=\lieq_1$.

The subalgebra $\liep$ admits real forms in $\mathfrak o(p+1,q+1)$ for each $p+q=2m$, all of which we shall denote again by $\liep$. Parabolic geometries of type $(\mathfrak o(p+1,q+1), P)$ correspond to \emph{semi-Riemannian conformal structures of type $(p,q)$ on manifolds of dimension $2m$}. The description of these structures and the possible types of strongly essential isotropies is completely analogous to the discussion in Section \ref{typeB}.

Under the isomorphism $D_3\cong A_3$, the algebras $\lieq_1$, $\liep$ and $\lieq_2$ correspond to $\liep(1)$, $\liep(2)$, and $\liep(3)$, respectively.  For the split real forms $\mathfrak{o}(3,3) \cong \mathfrak{sl}(4,\BR)$, this isomorphism in particular reflects the well-known relation between conformal geometry of signature $(2,2)$ and real $(2,2)$ almost Grassmannian structures. Because of the triality of the Dynkin diagram $D_4$, the class of geometries corresponding to $\liep$ and $\lieq$ in this rank are isomorphic; both give rise to conformal geometry in dimension $6$.    

Now assume $m = n-1 \geq 4$. Equip $\BR^{2n}$ with the standard inner product of signature $(n,n)$. 
The algebra $\lieq$ has a real form in $\mathfrak o(n,n)$, 
which we also denote by $\lieq$. To avoid the more complicated spinor description, realize $\lieq$ as the stabilizer in $\lieg$ of 
the highest weight line in the representation of self-dual $n$-forms on $(\BR^{n,n})^*$---that is, of the self-dual isotropic subspace 
$\mbox{span}\{ {\bf e}_1, \ldots, {\bf e}_n \} \subset \BR^{n,n}$.  Identifying the Levi factor $\lieg_0$ of $\lieq$ with the stabilizer of the complementary isotropic subspace $\mbox{span}\{ {\bf e}_{n+1}, \ldots, {\bf e}_{2n} \} \subset \bold{R}^{n,n}$ gives an identification of $\lieg$ with the $|1|$-graded Lie algebra
\begin{equation}\label{grading_spinorial_str}
\lieg=\lieg_{-1}\oplus\lieg_0\oplus\lieg_1=\Lambda^2\bold{R}^{n}\oplus\alggl(n,\bold{R})\oplus \Lambda^2\bold{R}^{n*},
\end{equation}
where $\lieq=\lieg_0\oplus\lieg_1$. The Lie bracket between elements in $\lieg_0$ and $\lieg_{-1}$ corresponds
to the action of $\alggl(n,\bold{R})$ on $\Lambda^2\bold{R}^{n}$. The bracket between elements $Z\in\lieg_1$ and $X\in\lieg_{-1}$ 
is $[Z,X]=-ZX$, where the right-hand side is the product of skew-symmetric matrices.

Parabolic geometries on $M^{n(n-1)/2}$ of type $(\mathfrak o(n,n), Q)$ are \emph{almost spinorial structures}. The standard homogeneous model is the variety of self-dual isotropic subspaces of $\BR^{n,n}$.  Such a geometry is given by an isomorphism $TM \cong \Lambda^2 E$, for $E$ an auxiliary rank-$n$ vector bundle. 
For more details about almost spinorial structures, see Section 4.1.12 of \cite{cap.slovak.book.vol1}.

The possible types of isotropies of higher-order fixed points of non-trivial strongly essential flows correspond to the
possible ranks $p=2\ell$ of nonzero skew-symmetric $n\times n$ matrices, $0 < \ell\leq \frac{n}{2}$; hence, there are precisely $\lfloor \frac{n}{2} \rfloor$ geometric types of isotropies of strongly essential flows.

\subsection{Smoothly isolated higher order zeroes}
\label{sec.isol_zeros}
Suppose $(B\stackrel{\pi}{\rightarrow} M,\omega)$ is a normal irreducible
parabolic geometry of type $(\lieg, P)$ admitting a nontrivial infinitesimal automorphism $\eta$
with a smoothly isolated higher-order zero at $x_0\in M$. Then its isotropy $\alpha\in T^*_{x_0}M$ 
satisfies $C(\alpha)=\{0\}$. Choose a point $b_0\in\pi^{-1}(x_0)$ to identify $\alpha$ 
with an element $Z\in\lieg_1$ and fix $X\in T(Z)$. By $(c)$ of Proposition \ref{eigenvalues_on_grading_components} 
the eigenspace decomposition of $\lieg_{-1}$ with respect to the action of 
$A=[Z,X]$ is of the form $\lieg_{-1}=\lieg_{-1}^{-2}\oplus\lieg_{-1}^{-1}$. Hence, $(d)$ and $(e)$ 
of Proposition \ref{eigenvalues_on_grading_components} imply that
the $\Bbb W^{[ss]}(A)=\{0\}$. Since $A \in \lieg_0$, we conclude that also $\widehat{\Bbb W}^{[ss]}(A)=\{0\}$.
If $\lieg$ is simple, Theorem \ref{thm.harmonic.vanishing} implies
that $\hat{\kappa}(x_0) = 0$ (in fact $\hat{\kappa}$ vanishes along the curves of $\mathcal{T}(\alpha)$).  Therefore we deduce from 
Corollary 2.14 of \cite{cap.me.parabolictrans} (compare \cite[Prop 2.2]{cap.me.parabolictrans}):

\begin{proposition}\label{Case_isolated_fixed_point} 
Suppose $(B\stackrel{\pi}{\rightarrow} M,\omega)$ is a normal irreducible 
parabolic geometry of type $(\lieg, P)$ with $\lieg$ simple.  Let $0 \neq \eta \in \mathfrak{inf}(M)$ 
with smoothly isolated higher-order zero at $x_0\in M$. Then there exists an open set $U\subset M$ such 
that $x_0\in \overline{U}$ on which the geometry is locally flat.  
\end{proposition}


Looking at the Lie brackets $\lieg_1\times\lieg_{-1}\rightarrow\lieg_0$ of the various geometries in Section
\ref{classification}, one can easily see which types of strongly essential zeroes are smoothly isolated. Then Proposition \ref{Case_isolated_fixed_point} yields the following Corollary, which in particular recovers some of the results on strongly essential flows on conformal structures in \cite{fm.champsconfs} 
(see also \cite{cap.me.proj.conf}) and Theorem 3.7 on almost quaternionic structures in \cite{cap.me.parabolictrans}.   

\begin{corollary} \label{Cor_isolated_zero}
Suppose $(B \stackrel{\pi}{\rightarrow} M,\omega)$ is a normal, irreducible, parabolic geometry. Let $0 \neq \eta\in\mathfrak{inf}(M)$ with a higher-order zero at $x_0\in M$.  Let $\alpha$ be the isotropy of $\eta$ at $x_0$. Then the higher-order zero $x_0$ is smoothly isolated, and the geometry locally flat on an open set $U\subset M$ with $x_0\in \overline{U}$, for the following model pairs $(\lieg,P)$ and geometric types of $\alpha$:
\begin{enumerate}
\item $A_n:$ $\lieg=\algsl(n+1,\bold{K})$ for $\bold{K}=\BR, \BC$ or $\bold{H}$, $n\geq 2$, and $\liep=\liep(1)$ as in Section \ref{typeA}
\item $A_n:$ $\lieg=\algsl(p+q,\bold{K})$ for $\bold{K}=\BR, \BC$ or $\bold{H}$, $2\leq p\leq q$, $\liep=\liep(p)$ as in Section \ref{typeA}, 
and $\mbox{rk}_{\bf K}(\alpha) = p$
\item $B_n/D_n:$ $\lieg=\mathfrak o(p+1,q+1)$ or $\mathfrak o(m+2,\BC)$, $m=p+q\geq 3$ and $pq=0$, and
$\liep$ as in Sections \ref{typeB} and \ref{typeD} 
\item $B_n/D_n:$ $\lieg=\mathfrak o(p+1,q+1)$ or $\mathfrak o(m+2,\BC)$, $m=p+q\geq 3$ and $pq\neq 0$, $\liep$ as 
in Sections \ref{typeB} and \ref{typeD}, and $\alpha$ non-isotropic
\item $C_n:$ $\lieg=\mathfrak s\mathfrak p(2n,\bold{K})$, $n\geq 3$, $\liep$ as in Section \ref{typeC}, and $\mbox{rk}_{\bf K}(\alpha) = n$.
\item $D_n:$ $\lieg=\mathfrak o(n,n)$ or $\mathfrak o(2n,\BC)$, $n\geq 5$, $\liep=\lieq$ as in Section \ref{typeD}, and 
$\mbox{rk}_{\bf K}(\alpha) = n$ for $n$ even, $\mbox{rk}_{\bf K}(\alpha) = n-1$ for $n$ odd.
\end{enumerate}
\end{corollary}

Note that $(1)$ of Corollary \ref{Cor_isolated_zero} includes projective, almost c-projective, and almost quaternionic structures. For projective structures, this curvature vanishing has been shown in \cite{nagano.ochiai.proj} to hold on a neighborhood of $x_0$.
We can prove the same improvement for almost c-projective structures and almost quaternionic structures.  

\subsubsection{Improvement for almost c-projective structures}
\label{sec.cproj.improvement}

Here we will appeal to some results and techniques of Section 2 of \cite{cap.me.parabolictrans}.
If $\lieg$ has a complex structure, then any semisimple 
element in $\lieg$ acts by $\BC$-linear maps, and has eigenspaces that are complex subspaces. Moreover, for any
$Z\in\lieg_1$ and any $X\in T(Z)$, the corresponding $\algsl_2$-triple $\{X, A, Z\}$  gives rise to a Lie algebra homomorphism
$\algsl(2,\BC)\rightarrow \lieg$, which integrates to a group homomorphism $\phi: SL(2,\BC)\rightarrow G$, for any choice of $G$ with Lie algebra 
$\lieg.$ For any $z,w\in\BC$ the following identity holds in  in $SL(2,\BC)$: 
\begin{equation*}
\left(\begin{array}{cc}
1& z\\
0&1
\end{array}\right)
\left(\begin{array}{cc}
1& 0\\
w&1
\end{array}\right)=
\end{equation*}
\begin{equation}\label{decomposition_SL2}
=\left(\begin{array}{cc}
1& 0\\
w(1+wz)^{-1}&1
\end{array}\right)\left(\begin{array}{cc}
1+wz& 0\\
0&(1+wz)^{-1}
\end{array}\right)\left(\begin{array}{cc}
1& z(1+wz)^{-1}\\
0&1
\end{array}\right).
\end{equation}

Fix $v \in \BC$.  For $s,t \in \BR$, set
\begin{equation*}
\hat{a}_t(s) = 
\left(
\begin{array}{cc}
1 + stv & 0 \\
0 & \frac{1}{1+stv}
\end{array}
\right)
\qquad
\hat{u}_t(s) = 
\left(
\begin{array}{cc}
1 & \frac{t}{1+stv} \\
0 & 1
\end{array}
\right).
\end{equation*}

These paths are smooth in $s$ for all $t \in \BR$ provided $v \notin \BR$, which we will assume below.  
Denote by $a_t(s)$ and $u_t(s)$ the images of these paths in $G_0$ and $P_+$ under the smooth embedding $\phi$, 
and set $p_t(s)= a_t(s) u_t(s)$.  
Setting $z=t$ and $w=sv$ the identity \eqref{decomposition_SL2} implies that the following product decomposition in $G$ holds
\begin{equation}\label{flow_in_model}
e^{tZ}e^{svX}=e^{c_t(s)X}a_t(s)u_t(s)\qquad s,t\in\BR, v\in\BC,
\end{equation}
where $c_t(s) = sv(1+stv)^{-1}$.  Note that $c_t(s) X$ traces a path in the abelian subalgebra $\BC X \subset \lieg_{-1}$.  
It follows that 
$$ \omega_G \left( \dds e^{c_t(s) X} \right) = c_t'(s) X  = \frac{v}{(1+stv)^2} X.$$
Differentiating equation (\ref{flow_in_model}) with respect to $s$ and evaluating with $\omega_G$ shows that
\begin{equation}
\label{eqn.vx.ode}
vX = \Ad p_t(s)^{-1} (c_t'(s) X) + \omega_G (p_t'(s)).
\end{equation}

Suppose now that $(B \stackrel{\pi}{\rightarrow} M,\omega)$ is a real, irreducible, normal, parabolic geometry of type $(\lieg, P)$ where $\lieg$ has a complex structure.
Assume also that $\eta\in\mathfrak{inf}(M)$ has a higher-order zero at $x_0\in M$. Choose $b_0\in\pi^{-1}(x_0)$ to identify the isotropy with $Z\in\lieg_1$. Then the proof of Proposition 2.1 of \cite{cap.me.parabolictrans} 
(see also \cite{fm.nilpconf}) can be modified in this setting as follows.  

The smooth $1$-form $\omega$ trivializes $TB \cong B \times \lieg$.  Consider the family of ODEs expressed in this trivialization as
\begin{eqnarray*}
E_t  \ : \ \beta_t'(s) & = & (\beta_t(s), c_t'(s) X) = \left( \beta_t(s), \frac{v}{(1+stv)^2} X \right) \\
\beta_t(0) & = & b_0
\end{eqnarray*}

Suppose $v\in\BC\setminus \BR$ has non-negative real part. Then on any compact interval $[0,\varepsilon]$ for $\varepsilon>0$, 
we have $|c_t'(s)|\leq |v|$ for all $s\in[0,\epsilon]$ and all $t\in [0,\infty)$. By the Picard--Lindel\"of Theorem (see \cite{teschl.ode.dyn.syst}), there exists $\varepsilon_0>0$ so that for any $t\geq 0$ the unique solution $\beta_t$ of $E_t$ is defined on 
$[0,\varepsilon_0]$. If $\mbox{Re } v <0$, just replace $t$ by $-t$.
Now Equation (\ref{eqn.vx.ode}) implies that in $B$,
\begin{equation}\label{flow_in_B}
\varphi_{\tilde\eta}^t\exp(b_0, svX)=\beta_t(s) a_t(s) u_t(s) \qquad \forall s \in [0,\epsilon_0]
\end{equation}

With respect to any positive definite inner product on $\lieg$, the norms $||c_t'(s) X || \rightarrow 0$ uniformly 
on compacts of $(0,\varepsilon_0]$ as $t \rightarrow \infty$.  In the corresponding smooth Riemannian metric on $TB \cong B \times \lieg$, 
the arc lengths $L(\beta_t) \rightarrow 0$. Therefore $\beta_t(s) \rightarrow b_0$ as $t \rightarrow \infty$, for any fixed $s \in [0,\varepsilon_0]$.
Then, as in Proposition 2.12 of \cite{cap.me.parabolictrans}, $a_t(s)$ is a holonomy path with attractor $b_0$ at $b(s)=\exp(b_0,svX)u_{\infty} (s)^{-1}$, where 
$$ u_{\infty}(s) = \lim_{t \rightarrow \infty} u_t(s) = \phi \left( \begin{array}{cc} 1 &   \frac{1}{sv} \\ 0 & 1 \end{array} \right)$$ 
The eigenspaces of $a_t(s)$ in $\widehat{\Bbb{W}}$ are independent of the choice of nonzero $v$, and so is the question of boundedness of $a_t(s)$ on these spaces.  Then as in Proposition 2.9 of \cite{cap.me.parabolictrans}, whenever $v\in\BC\setminus \BR$,
\begin{equation}
\hat{\kappa}(\exp(b_0, svX)) \in\widehat{\mathbb W}^{[st]}(A)\qquad \forall s.
\end{equation}
When $v \in \BR$, then replacing $s$ by $sv$ in Proposition 2.12 of \cite{cap.me.parabolictrans} and again applying Proposition 2.9 
gives the same stability along $\exp(b_0,sX)$, wherever the curve is defined.  We have proved:

\begin{proposition}\label{prop.complex.multiples}
Let $(M,B, \omega)$ be a real, irreducible, normal, parabolic Cartan geometry modeled on $(\lieg, P)$.  Suppose that $\lieg$ admits the structure of a complex Lie algebra.  
Let $0 \neq \eta \in \mathfrak{inf}(M)$ have higher order zero at $x_0$, and suppose the isotropy with respect to $b_0 \in \pi^{-1}(x_0)$ equals $Z$.  
Then, for any $X \in T(Z)$ and sufficiently small $s$, for $A = [Z,X]$,  
$$ \hat{\kappa}(\exp(b_0,svX)) \in \widehat{\mathbb W}^{[st]}(A)\quad \textrm{ for any } v\in \BC.$$ 
If $\hat{\kappa}(b_0) = 0$, then
$$\hat{\kappa}(\exp(b_0,svX )) \in \widehat{\mathbb W}^{[ss]}(A)\quad \textrm{ for any } v\in \BC$$
\end{proposition}
Whenever $\lieg$ is simple, then $\hat{\kappa}(b_0) = 0$ by Theorem \ref{thm.harmonic.vanishing}. 
From Proposition \ref{prop.complex.multiples} we now deduce Theorem \ref{thm.almost.cproj} about almost c-projective structures:

\begin{proof}[Proof of Theorem \ref{thm.almost.cproj}] 
Almost c-projective structures on $M^{2n}, n \geq 2$, are equivalent to real normal parabolic geometries modeled on $(\algsl(n+1,\BC),P)$, where $P$ is the stabilizer of complex line in $\BC^{n+1}$.  The Lie algebra $\lieg=\algsl(n+1,\BC)$ has a complex structure, and there is only one geometric type of strongly essential isotropy (see Subsection \ref{typeA}). By $(1)$ of Corollary \ref{Cor_isolated_zero}, any higher-order zero is necessarily smoothly isolated. 

Suppose $\eta\in\mathfrak{inf}(M)$ and $x_0M$ are as in the theorem, and let $Z$ be the isotropy with respect to $b_0 \in \pi^{-1}(x_0)$.  Then $C(Z)=\{0\}$, which means $\widehat{\mathbb W}^{[ss]}(A)=\{0\}$, as in the discussion at the beginning of Section \ref{sec.isol_zeros}, for any $X \in T(Z)$. Since $\lieg$ is a real simple Lie algebra, $\hat{\kappa}(x_0) = 0$ by Theorem 
\ref{thm.harmonic.vanishing}. So Proposition \ref{prop.complex.multiples} implies that $\hat{\kappa}$ vanishes along all curves of
the form $\pi(\exp(b_0, svX))$ for any $X\in T(Z)$ and $v\in \BC$. 

It remains only to show that these curves fill a dense subset of a 
neighborhood of $x_0$. For any $X\in T(Z)$ the eigenspace decomposition of $\lieg_{-1}$ with respect to $A$ is given by
$$\lieg_{-1}=\lieg_{-1}^{[-2]}\oplus\lieg_{-1}^{[-1]}=\BC X\oplus \ker(Z)$$
and $T(Z)=\{X+Y: Y\in\ker(Z)\}$. Since the set of nonzero complex multiples of $T(Z)$ equals $\lieg_{-1}\setminus~\ker(Z)$, 
it forms a dense open subset of $\lieg_{-1}$. Hence, the curves $\pi(\exp(b_0, svX))$ fill up a dense open subset of a neighborhood of $x_0$,
which implies the claim.
\end{proof}

\subsubsection{Improvement for almost quaternionic structures}

Almost quaternionic structures on $M^{4n}, n \geq 2,$ are reductions of structure group to $G_0=GL(n,\bold{H})\times _{\bold{Z}_2} Sp(1)$, comprising prequaternionic vector space isomorphisms $\bold{H}^n\cong\BR^{4n}$. They correspond to normal parabolic geometries modeled on $(\algsl(n+1,{\bf H}),P)$, where $P$ is the stabilizer of a quaternionic line in ${\bf H}^{n+1}$.  Our arguments are slightly more subtle, since the real Lie algebra $\lieg=\algsl(n+1,\bold{H})$ has no quaternionic structure.  From Section \ref{typeA}, there is only one geometric type of isotropy $\alpha$ of strongly essential flows on these geometries, all with $C(\alpha) = 0$ by $(1)$ of Corollary \ref{Cor_isolated_zero}.  

Suppose now that $(B \stackrel{\pi}{\rightarrow} M,\omega)$ is an irreducible, normal, parabolic geometry of type $(\algsl(n+1,\bold{H}), P)$, and assume $\eta\in\mathfrak{inf}(M)$ 
has a higher-order zero at $x_0\in M$. For a suitable choice of $b_0\in\pi^{-1}(x_0)$, the isotropy of $\eta$ will be the dual $Z$ of the first standard 
basis vector $\textbf{e}_1$ of $\bold{H}^{n}\cong\lieg_{-1}$. From \eqref{Grassmannian_grading}, $T(Z)=\textbf{e}_1+\ker(Z)$.  For any $X\in T(Z)$, the eigenspace decomposition of the action of $A=[Z, X]$ on $\lieg_{-1}$ is
\begin{equation}\label{eigenspace_decom_quatern}
\lieg_{-1}=\lieg_{-1}^{[-2]}\oplus\lieg_{-1}^{[-1]}= X\bold{H}\oplus \ker(Z).
\end{equation}
This shows that although elements of $\lieg_0$ in general do not 
act by ${\bf H}$-linear maps on $\lieg_{-1}$, the semisimple element $A\in\lieg_0$ 
does. The right $\bold{H}$-module spanned by $\{X, A, Z\}$ is thus contained in $\lieg$, and so is the 
$\algsl(2,\bold{H})$ subalgebra they generate. Hence, for any $X \in T(Z)$, the triple $\{X, A, Z\}$ induces a Lie algebra homomorphism
$\algsl(2,\bold{H})\rightarrow \lieg$, which locally integrates to a group homomorphism $\phi: SL(2,\bold{H})\rightarrow G$. 

For $z\in\BR$ and $w\in\bold{H}$ 
the identity \eqref{decomposition_SL2} holds in $SL(2,\bold{H})$. Hence the equality \eqref{flow_in_model} is still valid in $G=PGL(n+1,\bold{H})$ for $v\in \bold{H}$ and $s,t\in\BR$, and reads as
\begin{equation}\label{flow_in_model2}
e^{tZ}e^{Xsv}=e^{Xc_t(s)}a_t(s)u_t(s)\qquad s,t\in\BR, v\in\bold{H},
\end{equation}
where $a_t(s)$ and $u_t(s)$ are defined analogously as in the almost c-projective case.
As before, $X{\bf H}\subset \lieg_{-1}$ is an abelian subalgebra of $\lieg$, so we compute
$$\omega_G \left( \dds e^{Xc_t(s)} \right) =X c_t'(s)  =  X(v(1+st)^{-1} - sv (1+stv)^{-2} tv) = X v(1+stv)^{-2} $$
Define the curves $\beta_t(s)$ in $B$ as before, and note that the corresponding functions in the ODE are given by
$$F_t(y,s) = (y, X v(1+stv)^{-2} ).$$
As for almost c-projective structures, one deduces that for any $v\in\bold{H} \backslash \BR$ there exists $\varepsilon_0>0$ such that the family of curves 
$\beta_t$ is defined on $[0,\varepsilon_0]$ (As above, we restrict to $t \geq 0$ when $\mbox{Re } v > 0$ and $t \leq 0$ when $\mbox{Re } v < 0$).  In $B$,
\begin{equation}\label{flow_in_B_quaternionic}
\varphi_{\tilde\eta}^t\exp(b_0, Xsv)=\beta_t(s) a_t(s) u_t(s)\quad \forall s  \in [0,\varepsilon_0].
\end{equation}
The norms $||X c_t'(s) ||$ in $\lieg$ still tend uniformly to $0$, and we conclude that 
\begin{equation}\label{curvature_along_quaternoinic_curves}
\hat{\kappa}(\exp(b_0,Xsv)) \in \widehat{\Bbb W}^{[st]}(A)
\end{equation}
 for any $v \in {\bf H}$ (since it was previously known for $v \in \BR$ by \cite[Props 2.12, 2.9 ]{cap.me.parabolictrans}).  Theorem \ref{thm.harmonic.vanishing} gives ${\kappa}(b_0) = 0$, which, together with $\widehat{\Bbb{W}}^{[ss]}(A) = 0$ gives vanishing of $\hat{\kappa}$ along  all curves of the form $\pi(\exp(b_0,Xsv))$ for $X\in T(Z)$ and $v\in \bold{H}$. Nonzero quaternionic multiples of $T(Z)$ form an open dense subset 
 of $\lieg_{-1}$ by \eqref{eigenspace_decom_quatern}, so the following theorem is proved:
 
\begin{theorem}\label{thm.almost.quat}
Let $M^{4n}, n \geq 2,$ be a smooth almost quaternionic manifold, and let $0 \neq \eta\in\mathfrak{inf}(M)$ have a higher-order zero at $x_0$. Then there exists a neighborhood $U$ of $x_0$ on which the geometry is locally flat---that is, locally isomorphic to $\bold{PH}^n$ with its standard quaternionic structure. 
\end{theorem}

\subsection{Higher order zeros with maximal strongly fixed component}
\label{sec.max_CZ}
This section concerns the opposite extreme from isolated zeroes, namely,
higher-order zeroes for which the isotropy $\alpha$ has $C(\alpha)$ of maximal possible dimension. 

Let $(B \stackrel{\pi}{\rightarrow} M, \omega)$ be modeled on one of the geometries of Section \ref{classification}.  Suppose $\eta\in\mathfrak{inf}(M)$ has higher-order zero at $x_0 \in M$ with isotropy $Z$ with respect to $b_0 \in \pi^{-1}(x_0)$.  The dimension of $C(Z)$ is maximal if and only if, for $X \in T(Z)$ and $A = [Z,X]$, the dimension of $\lieg_{-1}^{[-2]}$ is minimal. Indeed, recall that $C(Z)\subset\lieg_{-1}$ equals the kernel of $\textrm{ad}(Z):\lieg_{-1}\rightarrow\lieg_0$. For almost all the geometries of Section \ref{classification}, elements of $\lieg_0$ can be identified with pairs of matrices $(A,B)$ such that the bracket $\lieg_1\times\lieg_{-1}\rightarrow\lieg_0$ is given by $[Z,X]=(ZX, -XZ)$.  The exception is in the conformal case, where the second matrix is $- XZ + \Bbb{I} (XZ)^t \Bbb{I}$.  For almost Lagrangean and almost spinorial structures, there is only the second term, $-XZ$. Then for $0 \neq Z\in\lieg_1$, the map $\textrm{ad}(Z):\lieg_{-1}\rightarrow\lieg_0$ has maximal kernel if and only if $Z$ is maximally degenerate---that is, for type $A_n$ or $C_n$, the ${\bf K}$-rank of $Z$ is $1$; for almost spinorial structures, $Z$ has rank $2$; and for pseudo-Riemannian conformal structures, $Z$ is null. (For Riemannian conformal structures, $C(Z)=\{0\}$ for any $Z\in\lieg_1$). From the descriptions of the gradings in Section \ref{classification}, it is easy to see that  
$T(Z)$ comprises the matrices $X\in\lieg_{-1}$ such that $ZX$ is the identity on $\textrm{im}(Z)$ and $XZ$ is the identity on $\textrm{im}(X)$, except in the conformal case 
where the description is a bit different and can be found in \cite{cap.me.proj.conf}. Then for all these geometries, all $X \in T(Z)$ have the same degree of degeneracy as $Z$. Moreover, when $Z\in\lieg_1$ is maximally degenerate, so $C(Z)$ has maximal dimension, then $\lieg_{-1}^{[-2]}$ is $1$-dimensional over $\bold{K}$, for any $A = [Z,X], X \in T(Z)$.  Again, Riemannian conformal structures are an exception; here $T(Z)$ consists of a single element and $A$ is twice the grading element, so $\lieg_{-1}=\lieg_{-1}^{[-2]}$ (see also \cite{cap.me.proj.conf}). Conversely, for all the geometries in Section \ref{classification}, the elements $Z\in\lieg_1$ yielding $\lieg_{-1}^{[-2]}$ of minimal dimension are the maximally degenerate isotropy classes, for which $C(Z)$ is of maximal dimension.

\begin{proposition}\label{C(Z)_max_prop} 
Suppose $(B \stackrel{\pi}{\rightarrow} M, \omega)$ is a normal, irreducible, parabolic geometry modeled on $(\lieg, P)$ with $\lieg$ simple. Let $0 \neq \eta\in\mathfrak{inf}(M)$ with higher-order zero at $x_0 \in M$.  Let $Z$ be the isotropy of $\eta$ with respect to $b_0\in\pi^{-1}(x_0)$.  Assume that $\lieg_{-1}^{[-2]}=\BR X$ for some $X\in T(Z)$. Then there exists an open set $U\subset M$ with $x_0\in\overline U$ 
on which the geometry is locally flat.
\end{proposition}
\begin{proof}
By Proposition 3.5(2a) of \cite{cap.me.parabolictrans}, for a sufficiently small neighborhood $U$ of $0$ in $C(Z)$ the set $N = \pi(\exp(b_0, U))$ 
defines a submanifold of $M$ that consists of higher-order zeroes of $\eta$; moreover, the isotropy of $\eta$ with respect to points of $\exp(b_0,U) \subset \pi^{-1}(N)$ equals $Z$. By Theorem \ref{thm.harmonic.vanishing}, $\hat{\kappa}$ vanishes on $N$ and along the family of curves $\mathcal{T}(\alpha)$ associated to each point in $N$. More precisely, let, where it is defined,
\begin{eqnarray*}
 \widetilde{\psi} & : & C(Z) \times T(Z) \times \BR \rightarrow B \\
\widetilde{\psi} & : & (Y,X,t) \mapsto \exp(\exp(b_0,Y), t X)
\end{eqnarray*}
and let $\psi = \pi \circ \widetilde{\psi}$.  Then Theorem \ref{thm.harmonic.vanishing} implies that  $\hat{\kappa}$ vanishes on the image of $\psi$.  

We wish to show that for any $X\in T(Z)$, there are a neighborhood $U \subseteq C(Z)$ of the origin and $\epsilon^* > 0$ such that $D_{(Y,X,\epsilon)} \psi$ 
is onto $T_{\psi(Y,X, \epsilon)} M$ for any $Y \in U$ and $0 < \epsilon < \epsilon^*$.  Then by the Inverse Function Theorem, the image of $\psi$ contains a 
neighborhood of $\psi(Y,X,\epsilon)$; varying $0 < \epsilon < \epsilon^*$ yields an open set in the image of $\psi$ containing $x_0$ in its closure.  To show surjectivity of $D_{(Y,X,\epsilon)} \psi$, it suffices to show that the image of $\omega \circ D_{(Y,X,\epsilon)} \widetilde{\psi}$ projects onto $\lieg / \liep$.  

First note that, for any $X \in T(Z)$, 
$$\omega \circ D_{(0,X,0)} \widetilde{\psi} (C(Z))  =  \omega \circ D_0(\exp_{b_0})(C(Z)) = C(Z).$$
Therefore, by making $Y$ and $\epsilon$ sufficiently small, we can make $\omega \circ D_{(Y,X,\epsilon)} \widetilde{\psi} (C(Z))$ arbitrarily close to $C(Z)$.  Moreover, 
$$ \omega \circ D_{(0,X,\epsilon)} \widetilde{\psi}(T_{\epsilon X} \BR T(Z)) = D_{\epsilon X}( \log_{\exp(b_0,\epsilon X)} \circ \exp_{b_0} ) (T_{\epsilon X} \BR T(Z)).$$
As $\epsilon \rightarrow 0$, the expression on the right approaches $T_X \BR T(Z)$, since the tangent space to the cone $\BR T(Z)$ is the same along the line $\BR X$.  
Thus for $Y$ and $\epsilon$ sufficiently small, we can make  $\omega \circ D_{(Y,X,\epsilon)} \widetilde{\psi}(T_{\epsilon X} \BR T(Z))$ arbitrarily close to $T_X \BR T(Z)$.   

By the characterization of $C(Z)$ in Proposition \ref{eigenvalues_on_grading_components} (c), we know that 
\begin{eqnarray*}
\lieg_{-1} & = & \BR X \oplus \lieg_{-1}^{[-1]} \oplus \lieg_{-1}^{[0]} \\
& = & \BR X \oplus \lieg_{-1}^{[-1]} \oplus C(Z)
\end{eqnarray*}
for any $X\in T(Z)$. Proposition \ref{eigenvalues_on_grading_components} (a) implies that this direct sum in turn equals
$$ \BR X \oplus \lieg_{-1}^{[-1]} \oplus C(Z) = \BR X \oplus \mbox{ad}(X) \lieg_0^{[1]} \oplus C(Z).$$
Recall from Proposition \ref{description_of_T(Z)} that $\mbox{Ad}(G_0^{[1]}) X = T(Z)$.  Therefore, 
$$T_X T(Z) = \mbox{ad}(\lieg_0^{[1]})X = \mbox{ad}(X)(\lieg_0^{[1]})$$
Now $T_X \BR T(Z) = \BR X \oplus T_X T(Z)$, so we conclude
$$ \lieg_{-1} = T_X \BR T(Z) \oplus C(Z),$$ which implies that $D_{(Y,X,\epsilon)} \psi$ is surjective.
\end{proof}

As a consequence of Proposition \ref{C(Z)_max_prop} and the discussion at the beginning of this subsection, we obtain the following Corollary, 
which in particular recovers some results of \cite{fm.champsconfs}
(see also \cite{cap.me.proj.conf}) on strongly essential Killing fields of pseudo-Riemannian 
conformal structures with null isotropy:

\begin{corollary} \label{Cor_C(Z)_max}
Suppose $(B \stackrel{\pi}{\rightarrow} M,\omega)$ is a normal irreducible
parabolic geometry of type $(\lieg, P)$ with $\lieg$ simple.  Let $0 \neq \eta\in\mathfrak{inf}(M)$ with higher-order zero at $x_0 \in M$, 
with isotropy $Z$ with respect to $b_0\in\pi^{-1}(x_0)$. Then $\lieg_{-1}^{[-2]}=\BR X$ for any $X\in T(Z)$, and the geometry is locally flat on an open set $U\subset M$ such 
that $x_0\in \overline{U}$, in the following cases:
\begin{enumerate}
\item $A_n:$ $\lieg=\algsl(p+q,\bold{R})$, $2\leq p\leq q$, $\liep=\liep(p)$ defined as in Section \ref{typeA} 
and $\mbox{rk}(\alpha) = 1$
\item $B_n$: $\lieg=\mathfrak o(p+1,q+1)$, where $m=p+q\geq 3$ and $pq\neq0$, $\liep$ as in Sections \ref{typeB} and \ref{typeD}, 
and $\alpha$ isotropic
\item $C_n$: $\lieg=\mathfrak s \mathfrak p(2n,\BR)$ for $n\geq 3$, $\liep$ as in Section \ref{typeC} and $\mbox{rk}(\alpha) = 1$.
\item $D_n$: $\lieg=\mathfrak o(n,n)$, $n\geq 5$, $\liep=\lieq$ defined as in Section \ref{typeD} and $\mbox{rk}(\alpha) = 2$. 
\end{enumerate}
\end{corollary}

Putting Corollary \ref{Cor_isolated_zero} and \ref{Cor_C(Z)_max} together yields Theorem \ref{thm.almost.grassmannian} for almost Grassmannian structures of type $(2,n)$, which provides an improvement of Theorem 3.1 of \cite{cap.me.parabolictrans}: 
\begin{proof}[Proof of Theorem \ref{thm.almost.grassmannian}]
Almost Grassmannian structures are equivalent to normal irreducible parabolic geometries of type $(\algsl(n+2,\BR), P(2))$. 
There are only two types of isotropies $\alpha$ of strongly essential infinitesimal automorphisms in this case, namely $\alpha$ 
having rank $1$ or $2$. Hence, the statement of Theorem  \ref{thm.almost.grassmannian}  follows directly from the Corollaries \ref{Cor_isolated_zero} and \ref{Cor_C(Z)_max}. 
\end{proof}

\section{Outlook: Irreducible and higher-graded parabolic geometries}
\label{sec.counterexs}

\subsection{Questions about automorphisms of irreducible parabolic geometries}

An obvious question stemming from our results is:

\begin{question}
For any irreducible parabolic geometry admitting a flow by strongly essential automorphisms, must the curvature vanish on an open set containing the fixed point in its closure?   
\end{question}

We conjecture that the answer is yes.  The shortfall between our techniques and this result is, roughly speaking, the difference between $T_X\BR T(Z)$ for $X \in T(Z)$ and $\lieg_{-1}^{[ss]}$---that is, for a fixed $X \in T(Z)$, there are in general other directions in $\lieg_{-1}^{[-2]}$ along which we are not able to show curvature vanishing.  These curves are, however, contracted under the flow.  A proof of this conjecture will likely require adapting the precise result of Proposition \ref{prop.s.to.the.k} and the ensuing proof of Theorem \ref{thm.harmonic.vanishing} to a coarser dynamical argument.

If the question above has a positive answer, this says that for any irreducible parabolic geometry, on any open set where the curvature does not vanish, any $\eta \in \mathfrak{inf}(M)$ is determined by its 1-jet.  Such a statement is the first step towards finding normal forms for arbitrary flows by automorphisms of these geometries, in the neighborhood of a fixed point (as in \cite{fm.champsconfs}).  A more ambitious question is thus:

\begin{question}
On an irreducible parabolic geometry in which the set of curvature vanishing has empty interior, are there flows by local automorphisms that are not linearizable?
\end{question}

Here we suppose the answer is yes. 

Lastly, we remark that the geometric structures underlying parabolic geometries can be defined with lower regularity: $C^0$ or $C^1$ semi-Riemannian conformal structures, for example, are well known.  These do not, however, determine a Cartan geometry, so our results do not apply to them.  With A. {\v C}ap, we have constructed a $C^1$ deformation of the flat $(2,n)$ Grassmannian model space, which still admits a strongly essential flow, with isotropy a rank one element in $\liep_+ \cong L(\BR^n, \BR^2)$.  We believe we can show this geometry is not locally flat on any nonempty open set, so it should be a counterexample to Theorem \ref{thm.almost.grassmannian} for low regularity $(2,n)$ almost Grassmannian geometries.

 \subsection{Submaximal path geometry of Kruglikov and The}
\label{sec.submaximal.example}

There are nowhere flat higher-graded parabolic geometries admitting a flow by strongly essential automorphisms.  Examples are given by the \emph{non-prolongation rigid} submaximal geometries constructed in \cite{kruglikov.the.submax}.  We briefly describe one family of such examples, due to Casey, Dunajski and Tod \cite{cdt.pair.odes} in the lowest dimension five, and proved submaximal by Kruglikov and The in higher odd dimensions.  The reader is referred to Section 5.3 of \cite{kruglikov.the.submax} for more details and references.  

The underlying manifold $M^{2m+1}$ carries a Cartan geometry modeled on the homogeneous space $\SL(m+2,\BR)/P_{1,2}$, where $P_{1,2}$ is the stabilizer of a flag consisting of a line contained in a plane in $\BR^{m+2}$.  This geometry is 2-graded with $\lieg_0 \cong \BR^2 \oplus \mathfrak{sl}(m,\BR)$.  Note that the flat model fibers over both ${\bf RP}^{m+1} = \SL(m+2,\BR)/P_1$ and the Grassmannian ${\bf Gr_{\BR}}(2,m+1) = \SL(m+2,\BR)/P_2$ of 2-planes in $\BR^{m+2}$.  

For $m \geq 3$, the geometric structure corresponding to these Cartan geometries encodes a system of second-order ODEs with $m$ dependent variables, $\{ x_1(t), \ldots, x_m(t) \}$, modulo point transformations, which are roughly changes of the coordinates $(t, x_1, \ldots, x_m)$.  For $m=2$, there is the same description, provided the homogeneity $1$ component of the harmonic curvature vanishes.  The flat model corresponds to the system $\ddot{x}_i = 0, i = 1, \ldots, m$, which, after compactifying in projective space, has solutions comprising the projective lines.  The model $\SL(m+2,\BR)/P_{1,2}$ is the projectivized tangent bundle of ${\bf RP}^{m+1}$.  

The results of \cite{cdt.pair.odes} and \cite{kruglikov.the.submax} establish the bound $m^2 + 5$ on the dimension of the symmetry algebra of a non-flat $(\SL(m+2,\BR),P_{1,2})$-geometry, and they exhibit a model for which this bound is achieved.  It corresponds to the system of ODEs $\ddot{x}_1 = \cdots = \ddot{x}_{m-1} = 0,$ $\ddot{x}_m = (\dot{x}_1)^3$. The full symmetry algebra acts transitively, with stabilizer isomorphic to $\lieq_1 \oplus \lieq_m$, where $\lieq_1$ and $\lieq_m$ are parabolic subalgebras of $\mathfrak{sl}(2,\BR)$ and $\mathfrak{sl}(m,\BR)$, respectively.  The intersection with $\lieg_0$ is $\BR \oplus \lieq_m$, leaving a 1-dimensional subspace of the stabilizer lying in $\lieg_1$.  We note that this geometry fibers over a $(2,m)$-almost-Grassmannian manifold, where this strongly essential flow becomes inessential.

\subsection{Questions on automorphisms of higher-graded parabolic geometries}

There are rigidity theorems for strongly essential automorphisms of some higher graded parabolic geometries.  For $C^\omega$, integrable CR structures of hypersurface type, Beloshapka \cite{beloshapka.str.ess.cr} and Loboda \cite{loboda.str.ess.cr} proved that strongly essential automorphisms can occur only on locally flat geometries.  Their proof involves rather elaborate calculations with Moser's normal forms, which are Taylor series expansions of automorphisms.  We wonder whether our approach could lead to a different proof of this result, perhaps also valid for $C^\infty$ structures.  Note that for strictly pseudo-convex CR structures---that is, those modeled on $\partial {\bf CH}^n$---there are theorems that any automorphism with a fixed point of a non-flat geometry is linearizable (see \cite{vitushkin.local.schoen}, \cite{kruzhilin.local.schoen}, \cite{cap.me.parabolictrans}).  

In light of the example in Section \ref{sec.submaximal.example} above, it seems that existence of strongly essential flows may not be the appropriate criterion for rigidity theorems in general.  A stronger hypothesis on a flow, which agrees with strongly essential for irreducible parabolic geometries, is having trivial 1-jet---that is, $\varphi^t(x_0) = x_0$ and $D_{x_0}\varphi^t = \mbox{Id}$ for all $t$.  This condition corresponds to the isotropy of the corresponding vector field $\eta$ lying in $\lieg^k \subset \liep_+$, the smallest subspace in the filtration defining $\liep$. 

\begin{question} On a parabolic geometry in which the set of curvature vanishing has empty interior, is any infinitesimal automorphism determined by its 1-jet at a point?
\end{question}

\bibliographystyle{amsplain}
\bibliography{karinsrefs}

\begin{tabular}{lll}
Karin Melnick  & \quad\qquad & Katharina Neusser \\
Department of Mathematics & \quad\qquad & Mathematical Sciences Institute  \\
University of Maryland & \quad\qquad & Australian National University \\
College Park, MD 20742 &\quad \qquad & Canberra, ACT 0200  \\
USA &\quad \qquad & Australia \\
karin@math.umd.edu &\quad \qquad & katharina.neusser@anu.edu.au\\
\end{tabular}

\end{document}